\documentclass[11pt, reqno]{amsart}
\usepackage[utf8]{inputenc}
\usepackage[english]{babel}
\usepackage{amsthm}
\usepackage{amsmath}
\usepackage{amsfonts}
\usepackage{amssymb}
\usepackage{hyperref}
\makeatletter
\newcommand{\maybeusepackage}[1]{\IfFileExists{#1.sty}{\usepackage{#1}}{}}
\makeatother
\maybeusepackage{mathrsfs}
\maybeusepackage{mathtools}
\maybeusepackage{xcolor}
\maybeusepackage{csquotes}
\theoremstyle{plain}
\newtheorem{theorem}{Theorem}[section]
\newtheorem{lemma}[theorem]{Lemma}
\newtheorem{corollary}[theorem]{Corollary}
\newtheorem{proposition}[theorem]{Proposition}
\theoremstyle{definition}
\newtheorem{definition}[theorem]{Definition}

\newtheorem{remark}[theorem]{Remark}

\theoremstyle{plain}
\newtheorem*{theoremA}{Theorem A}
\newtheorem*{theoremB}{Theorem B}
\newtheorem*{theoremC}{Theorem C}
\numberwithin{equation}{section}

\newcommand{\R}{\mathbb{R}}

\newcommand{\id}{\mathrm{id}}
\newcommand{\Ad}{\mathrm{Ad}}
\newcommand{\eexp}{\overrightarrow{\exp}}

\newcommand{\rank}{\mathrm{rank}}
\newcommand{\Span}{\mathrm{span}}

\newcommand{\im}{\mathrm{Im}}

\newcommand{\vect}{\mathrm{Vec}}
\newcommand{\lie}{\mathrm{Lie}}
\newcommand{\be}{\begin{equation}}
\newcommand{\ee}{\end{equation}}
\newcommand{\Len}{\mathsf{L}}
\newcommand{\bv}{\mathbf v}
\newcommand{\bV}{\mathbf V}
\newcommand{\bw}{\mathbf w}
\title{Strong Legendre regularity and short-time analysis of abnormal curves}
\author[F.~Boarotto]{Francesco Boarotto}
\email{francesco.boarotto@gmail.com}

\begin{document}
\begin{abstract}
We study strictly abnormal sub-Riemannian extremals satisfying the Goh condition and
a strong (i.e.\ coercive) generalized Legendre condition, imposed only at the
Lebesgue points of the reference control. We prove that such a control is smooth, that
the projected extended Hessian of the end-point map is uniformly coercive on the
admissible set of variations for short times, and that the abnormal covector gives a
quantitative separation of the end-point of the reference arc from the image of the
admissible set under the nonlinear end-point map. As a consequence the reference arc is
rigid and, for short times, it is a length minimizer: every admissible curve with the
same end-points and no greater length is a reparametrization of the reference arc.
In particular, among controls of constant speed the reference control is the unique
length minimizer. All the
estimates are uniform on the admissible set, and no assumption is made on the dimension
of the manifold, on the growth vector of the distribution, or on the corank of the
differential of the end-point map.
\end{abstract}
\maketitle

\section{Introduction}\label{sec:intro}

\subsection{General framework}
One of the central objects of sub-Riemannian theory is the end-point map
\cite{ASBook, ABB20, Mont02}. Our current understanding of this map is limited and
barely goes beyond its second-order expansion \cite{BMS24, Joz26, BMP20}, one of the main
issues being the lack of a solid theory of high-order open mapping theorems.
Nonetheless, the iterated integrals appearing in the asymptotic expansion of the
end-point map encode, in principle, all the geometric information one needs, and their
understanding is probably necessary for any significant advance in the theory.

At the second-order level the picture is dramatically different, and the end-point map is
nowadays very well understood. Starting from the seminal work of Agrachev \cite{AQ88}, an
intrinsic theory has been developed, in which the Hessian of the end-point map is the
crucial tool to study candidate length-minimizing curves whose first differential is not
surjective, that is, abnormal curves. Within this framework, the Goh condition
\cite{Goh88} and the generalized Legendre condition play a crucial role in establishing
necessary conditions that any candidate optimal control must satisfy. The generalized
Legendre condition is a condition on the non-negativity of a family of quadratic forms.

In this paper we address the case in which these quadratic forms are essentially
coercive, meaning that they are essentially uniformly bounded away from zero. We prove
that in this case the reference arc is, for short times, a length minimizer, and that
every competitor with the same end-points and no greater length is a reparametrization
of it. In particular, the reference control is locally optimal with respect to the
$L^2$-topology on the space of controls (equivalently, the $W^{1,2}$-topology on the
space of trajectories), and strictly so up to reparametrizations. This line of research extends the results of \cite{AS96,
AS95}, where local optimality was obtained either with respect to the
$L^\infty$-topology on the space of controls, or restricted to rank-two distributions in
the $L^2$ case. Even the form in which coercivity is assumed throughout this paper is
slightly weaker than the one used in the two references above, since we only require the
inequality at the Lebesgue points of the reference control.

\subsection{Setting of the problem}
A sub-Riemannian manifold is a triplet $(M,\Delta,g)$, where $M$ is a smooth (i.e.,
$C^\infty$) manifold of dimension $\dim(M)=n$, $\Delta\subset TM$ is a distribution of
constant rank $2\le d<n$, and $g$ is a metric on $\Delta$. In a neighborhood $U\subset M$
of any point $q\in M$ there exist vector fields $X_1,\dots,X_d\in\vect(U)$ such that
$\Delta=\Span\{X_1,\dots,X_d\}$ on $U$. Our considerations being local, we assume $U=M$
throughout the paper. We also assume that $\Delta$ satisfies the H\"ormander condition
\begin{equation}\label{eq:Hormander_condition}
	\lie\left\{X_1,\dots,X_d\right\}(p)=T_pM, \qquad p\in M,
\end{equation}
i.e., that $\Delta$ is non-holonomic. An admissible curve is a curve
$\gamma\in W^{1,2}([0,1],M)$ whose derivative $\dot\gamma$ belongs to $\Delta_\gamma$
a.e.\ on $[0,1]$, meaning that
\begin{equation}\label{eq:horizontal_curve}
	\dot\gamma(t)=\sum_{i=1}^d u_i(t)X_i(\gamma(t)),
	\qquad\text{for a.e.\ }t\in[0,1],
\end{equation}
for a unique $u\in L^2([0,1],\R^d)$. For brevity we adopt the short-hand notation
$X_w:=\sum_{i=1}^dw_iX_i$ for $w\in\R^d$. If the frame $X_1,\dots,X_d$ is orthonormal
with respect to $g$, the length of an admissible curve $\gamma$ is the $L^1$-norm of its
associated control; it is invariant under reparametrization.

We study admissible curves stemming from $q$, i.e., solutions to
\eqref{eq:horizontal_curve} with $\gamma(0)=q$. The time-$t$ end-point map
$F_q^t:L^2([0,1],\R^d)\to M$ is defined by $F_q^t(u)=\gamma_u(t)$. A control
$u\in L^2([0,1],\R^d)$ is called singular if $d_uF^1_q$ is not surjective.

It is customary to study minima of the length through the energy functional, and to
introduce the extended end-point map
\begin{equation}\label{eq:extended_endpoint}
	F_J:L^2([0,1],\R^d)\to M\times\R,
	\qquad
	F_J(u)=\left(F^1_q(u),J(u)\right),
	\qquad
	J(u):=\frac12\|u\|^2_{L^2([0,1],\R^d)}.
\end{equation}
If $\lambda_0=0$ for every $(\lambda,\lambda_0)\in\im(d_uF_J)^\perp$, the singular
control $u$ is called strictly singular. The main difficulties in sub-Riemannian
geometry arise from strictly singular controls which, as discovered in the
ground-breaking work of Montgomery \cite{Mon94}, can indeed be length minimizers. Since
$\lambda_0=0$ for every covector orthogonal to the image of $d_uF_J$, the differential of
the energy plays no role in the linearized constrained problem, and the analytic question
becomes the openness of the end-point map $F^1_q$ restricted to the set of curves whose
control $\tilde u$ satisfies the energy constraint $J(\tilde u)\le J(u)$. The covector
$\lambda$ is orthogonal to $\im d_uF_q^1$, and this is precisely why the second-order
analysis becomes essential.

Two fundamental geometric conditions arise from this analysis. The Goh condition
states that a strictly singular length minimizer $\gamma$ admits an adjoint curve of
covectors $[0,1]\ni t\mapsto\lambda(t)\in T^*_{\gamma(t)}M$ such that
\begin{equation}\label{eq:Goh_intro}
	\left\langle\lambda(t),\left[X_i,X_j\right](\gamma(t))\right\rangle=0,
	\qquad 0\le t\le1,\quad i,j=1,\dots,d;
\end{equation}
we may therefore think of the Goh condition as the second-order counterpart of the
first-order conditions deriving from the Pontryagin Maximum Principle (PMP, see
\cite{PMP}), namely $\langle\lambda(t),X_i(\gamma(t))\rangle=0$ for every $t\in[0,1]$ and
every $i=1,\dots,d$. The companion necessary condition of second order is the
generalized Legendre condition
\begin{equation}\label{eq:legendre_intro}
	Q^t(w)=\left\langle\lambda(t),
	\left[\left[X_{u(t)},X_w\right],X_w\right](\gamma(t))\right\rangle\ge0,
	\qquad\text{for a.e. }t\in[0,1],\quad w\in u(t)^\perp,
\end{equation}
which holds for the same adjoint curve for which \eqref{eq:Goh_intro} is true. Two
comments are in order. First, the bilinear form
\[
	b^t(v,w):=\left\langle\lambda(t),
	\left[[X_{u(t)},X_v],X_w\right](\gamma(t))\right\rangle
\]
is symmetric because of \eqref{eq:Goh_intro}, so that \eqref{eq:legendre_intro} is a
sign-definiteness condition on the diagonal $Q^t(w)=b^t(w,w)$. Second, since $Q^t$
contains $u(t)$ in its expression, it is defined only when $t$ is a Lebesgue point of
$u$, hence a.e.\ on $[0,1]$.

\subsection{Main results}\label{ssec:mainresults}
Throughout the paper $u$ is a strictly singular control, parametrized by arclength, along
which the Goh condition \eqref{eq:Goh_intro} holds for a fixed nonzero adjoint curve
$t\mapsto\lambda(t)$, and along which the Legendre form is coercive: there is
$c>0$ with
\begin{equation}\label{eq:strongLeg_intro}
	Q^t(w)\ge c|w|^2
	\qquad\text{for every }w\in u(t)^\perp
	\text{ and every Lebesgue point }t\text{ of }u.
\end{equation}
Condition \eqref{eq:strongLeg_intro} is the strong generalized Legendre condition;
the constant $c$ depends on the normalization of $\lambda$, which is fixed once and for
all. We write $q(t):=F^t_q(u)$ for the reference trajectory. The results of the paper are the
following three theorems.

\begin{theoremA}[Regularity; Proposition~\ref{prop:ucont} and
Proposition~\ref{prop:smooth} in Section~\ref{sec:regularity}]
	Let $d\ge2$ and assume \eqref{eq:Goh_intro} and \eqref{eq:strongLeg_intro}. Then $u$
	admits a continuous representative, and this representative is smooth on $[0,1]$.
	Moreover, for every $t\in[0,1]$ the control value $u(t)$ is the unique unit vector
	$w$ in the kernel of the differentiated Goh operator $B(z(t))$ of
	\eqref{eq:Bmap} such that $T_{z(t)}(w,y,y)>0$ for every non-zero $y\perp w$ (see
	\eqref{eq:trilinear}), and it depends smoothly on the state-costate pair
	$z(t)=(q(t),\lambda(t))$.
\end{theoremA}

\begin{theoremB}[Coercivity and separation; Theorem~\ref{thm:perturbedcoercivity} and
Theorem~\ref{thm:separation} in Sections~\ref{sec:hessian}--\ref{sec:separation}]
	Under the same hypotheses there exist $t_*\in(0,1]$ and constants $\eta>0$,
	$C_4>0$, all independent of $t$, such that for every $t\in(0,t_*]$ and every
	variation $\bv\in L^2([0,t],\R^d)$ satisfying the admissibility constraint
	\eqref{eq:conepointwise},
	\[
		\left\langle\ell_t,G^t_{q(t)}(\bv)\right\rangle
		\ \ge\ \eta\left\|\bV\right\|_{\mathcal H}^2
		-C_4\left|G^t_{q(t)}(\bv)\right|^2 ,
	\]
	where $G^t_{q(t)}$ is the end-point map read in the adapted chart of
	Lemma~\ref{lem:adapted}, $\ell_t$ is the abnormal covector in that chart,
	$\bV$ is the primitive of $\bv$, and $\|\cdot\|_{\mathcal H}$ is the norm
	\eqref{eq:Hnorm}.
\end{theoremB}

\begin{theoremC}[Rigidity and short-time minimality; Corollary~\ref{cor:rigidity} and
Theorem~\ref{thm:minimality} in Section~\ref{sec:separation}]
	Under the same hypotheses there exists $t_*\in(0,1]$ such that, for every
	$t\in(0,t_*]$, the arc $\gamma_u|_{[0,t]}$ is rigid and minimizes the length among all
	admissible curves joining $q$ to $q(t)$: if $\gamma_{\tilde u}$ is an admissible curve
	joining $q$ to $q(t)$ with $\ell(\gamma_{\tilde u})\le t$, then $\ell(\gamma_{\tilde u})=t$
	and $\gamma_{\tilde u}$ coincides with $\gamma_u|_{[0,t]}$ up to reparametrization. In
	particular, $u|_{[0,t]}$ is the unique length-minimizing control of constant speed
	joining $q$ to $q(t)$.
\end{theoremC}

Section~\ref{sec:endpoint} collects the chronological formalism and the second-order
conditions; Section~\ref{sec:regularity} contains the algebraic preliminaries, the proof
of Theorem~A, and the construction of the adapted frame and chart;
Section~\ref{sec:hessian} is devoted to the extension and to the coercivity of the
Hessian; Section~\ref{sec:separation} contains the third-order expansion of the
end-point map, the separation theorem, and the deduction of Theorem~C. The concluding
Subsection~\ref{ssec:strength} discusses the strength and the limitations of the results.

\subsection{Why a single direction, and a single covector, suffice}\label{ssec:onedir}
It is worth explaining at once the architecture of the proof, because it is what makes
the argument short and insensitive to the corank of the end-point map.

An exact competitor is a solution of the vector equation
$G^t_{q(t)}(\bv)=0$ in $\R^n$. In order to exclude a competitor, however, it
suffices to exhibit one scalar consequence of this equation which is violated:
one scalar functional which vanishes on the reference control, which is built out of the
constraint, and which, up to a quadratic error in the end-point displacement, is strictly positive
on every non-zero admissible variation. The corresponding inequality then separates
$q(t)$ from the image of the admissible set under the nonlinear end-point map, and no
competitor can reach $q(t)$. This is the
elementary asymmetry between proving that a point is in the image of a map -- for
which one must solve the full vector equation, and where the dimension of the cokernel
does matter -- and proving that it is not.

The scalar functional we use is the component along the abnormal covector $\lambda(t)$,
that is, the displacement measured along one single codirection. The reason why this
particular direction is available, and the others are not, is precisely the Goh
condition. Condition \eqref{eq:Goh_intro} is an existence statement: it asserts that
at least one abnormal lift annihilates all the brackets $[X_i,X_j]$ along the
curve. It does not assert --- and in general it is false --- that every element of the
annihilator
$\Lambda_t:=\left(\im d_0G^t_{q(t)}\right)^\perp$ enjoys the same property. All the
structural information exploited in this paper is attached to that single covector: the
linear constraint \eqref{eq:linearGoh} which forces the smoothness of $u$, the vanishing
\eqref{eq:cruc_vanishing} which eliminates the first component of the variation from the
Hessian, the reduction \eqref{eq:hessGnoV1}, and the Legendre form \eqref{eq:slg1}. For a
covector $\mu\in\Lambda_t$ not satisfying Goh, the reduced Hessian
$\mu\widehat{H^t}$ of \eqref{eq:extensHessian} is not even well defined.

\subsection{Nice abnormals and the one-dimensionality of $\ker B(z)$}
\label{ssec:nice}
For rank-two distributions the classical hypothesis under which abnormal minimality is
studied is that the abnormal extremal be nice, i.e.\ that
\begin{equation}\label{eq:nice}
	\lambda(t)\in\left(\Delta^2_{q(t)}\right)^\perp\setminus
	\left(\Delta^3_{q(t)}\right)^\perp ,
	\qquad t\in[0,1],
\end{equation}
where $\Delta^1:=\Delta$ and $\Delta^{k+1}:=\Delta^k+[\Delta,\Delta^k]$; see
\cite[Chapter~12]{ABB20} and \cite{AS95}. The first membership in \eqref{eq:nice} is
exactly the Goh condition; the second one is a non-degeneracy requirement at the level of
brackets of length three.

It is instructive to compare \eqref{eq:nice} with the algebraic mechanism used here. The
differentiated Goh condition is encoded by the linear operator
$B(z):\R^d\to\Lambda^2(\R^d)^*$ of \eqref{eq:Bmap}, and the whole regularity theory of
Section~\ref{sec:regularity} rests on the fact that
\begin{equation}\label{eq:onedim}
	\ker B(z(t))=\Span\{u(t)\},
	\qquad t\in[0,1],
\end{equation}
which we derive from the strong Legendre condition (Lemma~\ref{lem:rank}). When
$d=2$ the space $\Lambda^2(\R^2)^*$ is one-dimensional and
$(B(z)a)_{12}=\langle\lambda,[X_a,[X_1,X_2]](q)\rangle$, so that $B(z)\neq0$ if and only
if $\lambda\notin(\Delta^3_q)^\perp$, and in that case $\ker B(z)$ is automatically a
line. Thus in rank two \eqref{eq:onedim} and the niceness condition \eqref{eq:nice}
coincide, and both are implied by strong Legendre (which is stronger than both).

In higher rank the analogy breaks down, and this is worth recording. The target
$\Lambda^2(\R^d)^*$ has dimension $d(d-1)/2$, so \eqref{eq:onedim} is a statement of
maximal rank for $B(z)$, i.e.\ $\rank B(z)=d-1$: it is a codimension-$(d-1)$ condition on
the bracket structure of length three, not the non-vanishing of a single bracket.
Requiring merely $\lambda\notin(\Delta^3_q)^\perp$, i.e.\ $B(z)\ne0$, leaves $\ker B(z)$
of dimension up to $d-1$, and neither the selection of $u(t)$ as a distinguished
direction nor the bordered invertibility of Lemma~\ref{lem:sign} -- hence neither the
regularity of $u$ nor the subsequent analysis -- survives. In other words, the natural
higher-rank substitute for niceness is not a condition on one bracket but the
non-degeneracy \eqref{eq:onedim}, and coercivity of the Legendre form is the most natural
geometric hypothesis producing it. See Remark~\ref{rmk:niceanalogy}.

\subsection{Comparison with the literature}
The present argument belongs to the second-variation tradition for singular and abnormal
extremals, continuing the analysis begun by Agrachev and Sarychev \cite{AS95,AS96}. There
is also a close connection with the no-gap second-order conditions established by
Dmitruk for control-affine problems without control constraints
\cite{Dmitruk2008,Dmitruk2026}. In that setting the original second variation is
automatically degenerate; after a Goh-type transformation it becomes a classical,
possibly non-degenerate quadratic form, whose positivity yields necessary and sufficient
conditions for weak minimality of quadratic order. The common principle is that a
degenerate singular second variation becomes coercive after passing to suitable primitive
variables --- although here we recover coercivity by treating all non-Legendre terms as
perturbative remainders, rather than by solving the associated Jacobi equation and
locating the first conjugate time.

The connection with Jacobi fields \cite{AgrachevBeschastnyi2020,
AgrachevBeschastnyi2022} is conceptually important as well. In the classical theory,
Jacobi fields describe the kernel of the second variation and conjugate times mark the
loss of coercivity. The transport identity of Lemma~\ref{lem:transport} identifies the
principal part of the reduced Hessian with the Legendre form along the reference
extremal; thus, on sufficiently short intervals, the coercivity estimate of
Theorem~\ref{thm:perturbedcoercivity} may be viewed as a quantitative ``no conjugate
point'' statement. Lastly, we point out the recent works \cite{Joz23,Joz26}, where new
second-order necessary conditions are derived through a careful expansion of the
end-point map in an adapted system of coordinates.

\section{End-point map analysis}\label{sec:endpoint}

\subsection{Elements of chronological calculus}
We adopt the formalism of chronological calculus introduced by Agrachev and Gamkrelidze
in \cite{AG78}, and explained in greater detail in \cite[Chapter~3]{ASBook}. The point of
view adopted within this framework consists in identifying points of $M$, vector fields,
and diffeomorphisms of $M$ with operators on $C^\infty(M)$. The symbol $\circ$ denotes
composition of operators.

\begin{definition}
	Let $0<t\le1$. The time-$t$ end-point map relative to the vector fields
	$X_1,\dots,X_d\in\vect(M)$ is the map $F^t_q:L^2([0,1],\R^d)\to M$ given by
	\begin{equation}\label{eq:endp}
		F^t_{q}(u)=q\circ\eexp\int_0^tX_{u(\tau)}\,d\tau .
	\end{equation}
\end{definition}

We assume that the Cauchy problem for $X_{u(t)}$ has a solution defined on the whole
interval $[0,1]$, and we set $q(t):=F^t_q(u)$. Dropping the dependence on the base point,
the map $F^t(\cdot,u):M\to M$, $F^t(q,u):=F^t_q(u)=q(t)$, is a diffeomorphism of $M$ and
an instance of a right-chronological exponential, obeying
\begin{equation}\label{eq:right_chronological_exp}
	\dot F^t(\cdot,u)=F^t(\cdot,u)\circ X_{u(t)} .
\end{equation}
We fix a point $q\in M$ and a reference control $u\in L^2([0,1],\R^d)$, which we assume
from the very beginning to be strictly singular and parametrized by arclength,
\[
	|u(t)|=1,\qquad\text{for a.e.\ }t\in[0,1].
\]

Let $0\le t\le1$ and let a variation $v\in L^2([0,t],\R^d)$ be given. We want to study
$F^t_q(u+v)$ as a perturbation of the reference trajectory driven by $u|_{[0,t]}$. By the
variation of constants formula \cite[Section~2.7]{ASBook},
\begin{equation}\label{eq:varfor}
	F^t_{q}(u+v)=G^t_{q(t)}(v)=q(t)\circ\eexp\int_0^tX^{u,\tau}_{v(\tau)}\,d\tau ,
\end{equation}
where
\[
	X^{u,\tau}_{v(\tau)}:=\Ad\left(\eexp\int_t^\tau X_{u(\theta)}\,d\theta\right)X_{v(\tau)} .
\]
By the above definition, for every $v\in\R^d$ one has
\begin{equation}\label{eq:derXut}
	\dot X^{u,\tau}_{v}=\Ad\left(\eexp\int_t^\tau X_{u(\theta)}\,d\theta\right)
	\left[X_{u(\tau)},X_v\right] .
\end{equation}
Like the time-$t$ end-point map, $G^t_{q(t)}$ satisfies a relation of the form
\eqref{eq:right_chronological_exp}, and therefore admits an explicit Volterra-type
expansion: for every $N\ge1$,
\begin{equation}\label{eq:Volterra_expansion}
	\begin{aligned}
		G_{q(t)}^t(v)&=q(t)\circ\left(\id+\sum_{k=1}^{N}\int_{\Sigma_k(t)}
		X^{u,\tau_k}_{v(\tau_k)}\circ\dots\circ X^{u,\tau_1}_{v(\tau_1)}
		\,d\tau_k\cdots d\tau_1\right)\\
		&+\int_{\Sigma_{N+1}(t)}q_v(\tau_{N+1})\circ X^{u,\tau_{N+1}}_{v(\tau_{N+1})}
		\circ\dots\circ X^{u,\tau_1}_{v(\tau_1)}\,d\tau_{N+1}\cdots d\tau_1 ,
	\end{aligned}
\end{equation}
where $q_v(\tau):=G^{\tau}_{q(t)}(v)$ and, for $k\ge1$ and $0\le t\le1$,
\begin{equation}\label{eq:simplex}
	\Sigma_k(t):=\left\{(\tau_k,\dots,\tau_1)\in\R^k\ \middle|\
	0\le\tau_k\le\dots\le\tau_1\le t\right\}
\end{equation}
denotes the standard $k$-dimensional simplex.

\begin{remark}[Orientation of the simplex]\label{rmk:simplex}
	The ordering in \eqref{eq:simplex} is the one induced by the right-chronological
	exponential: in every iterated integral appearing below, the leftmost operator
	carries the smallest time. This convention is in force throughout the paper; it
	fixes the sign of the second-order term \eqref{eq:sdif} and it is the one under which
	$\tau_3$ is the innermost integration variable in the third-order remainder
	\eqref{eq:remainderintegral}.
\end{remark}

The first differential of $G^t_{q(t)}$ at $v=0$ is
\begin{equation}\label{eq:fdif}
	d_0G^t_{q(t)}(v)=q(t)\circ\int_0^tX^{u,\tau_1}_{v(\tau_1)}\,d\tau_1,
	\qquad v\in L^2([0,t],\R^d).
\end{equation}
The action of a diffeomorphism $P$ on a vector $v$ is written, in chronological notation,
as $P_*v=v\circ P$. Using extension by zero to identify $L^2([0,\tau],\R^d)$ with a
subspace of $L^2([0,t],\R^d)$ for $0\le\tau\le t\le1$, the first differential obeys the
composition rule
\begin{equation}\label{eq:first_differential_composition}
	\begin{aligned}
		\left(\eexp\int_\tau^tX_{u(\theta)}\,d\theta\right)_*d_0G^\tau_{q(\tau)}(v)
		&=\left(q(\tau)\circ\int_0^\tau X^{u,\tau_1}_{v(\tau_1)}\,d\tau_1\right)
		\circ\int_\tau^tX_{u(\theta)}\,d\theta\\
		&=q(t)\circ\Ad\left(\eexp\int_t^\tau X_{u(\theta)}\,d\theta\right)
		\circ\int_0^\tau X^{u,\tau_1}_{v(\tau_1)}\,d\tau_1\\
		&=q(t)\circ\int_0^t\Ad\left(\eexp\int_t^{\tau_1}X_{u(\theta)}\,d\theta\right)
		X_{v(\tau_1)}\,d\tau_1=d_0G^t_{q(t)}(v),
	\end{aligned}
\end{equation}
which shows that, once the transported image is identified with a subspace of the final
one, the image of $d_0G^t_{q(t)}$ can only increase with $t$. In particular, since $u$ is
strictly singular, the first differential
$d_0G^t_{q(t)}:L^2([0,t],\R^d)\to T_{q(t)}M$ fails to be surjective for every
$0<t\le1$. For every nonzero covector $\lambda(1)\in T^*_{q(1)}M$ orthogonal to the image of
$d_0G^1_{q(1)}$, the curve $[0,1]\ni t\mapsto\lambda(t)$ defined by
\begin{equation}\label{eq:covectcurve}
	\lambda(t):=\left(\eexp\int_1^tX_{u(\tau)}\,d\tau\right)^*\lambda(1)\in T^*_{q(t)}M
\end{equation}
is orthogonal to $\im d_0G^t_{q(t)}$, since by \eqref{eq:first_differential_composition}
\[
	\left\langle\lambda(t),d_0G^t_{q(t)}(v)\right\rangle
	=\left\langle\lambda(1),d_0G^1_{q(1)}(v)\right\rangle=0,
	\qquad t\in[0,1],\quad v\in L^2([0,t],\R^d).
\]
Since $t\mapsto q(t)$ and $t\mapsto\lambda(t)$ are absolutely continuous, for every smooth
vector field $Y$ on $M$ the function $t\mapsto\langle\lambda(t),Y(q(t))\rangle$ is
absolutely continuous and, by \eqref{eq:covectcurve} and the computation leading to
\eqref{eq:derXut} (with $X_v$ replaced by $Y$),
\begin{equation}\label{eq:derpairing}
	\frac{d}{dt}\left\langle\lambda(t),Y(q(t))\right\rangle
	=\left\langle\lambda(t),\left[X_{u(t)},Y\right](q(t))\right\rangle
	\qquad\text{for a.e.\ }t\in[0,1].
\end{equation}

\subsection{Second-order conditions}
The second differential is twice the second iterated integral in
\eqref{eq:Volterra_expansion}, that is
$\frac12d_0^2G^t_{q(t)}(v)=q(t)\circ\int_{\Sigma_2(t)}X^{u,\tau_2}_{v(\tau_2)}
\circ X^{u,\tau_1}_{v(\tau_1)}\,d\tau_2d\tau_1$. Following the computations of
\cite[Section~3]{BMP20},
\begin{equation}\label{eq:Second_order_differential}
	d_0^2G_{q(t)}^t(v)=q(t)\circ\int_{\Sigma_{2}(t)}
	\left[X^{u,\tau_2}_{v(\tau_2)},X^{u,\tau_1}_{v(\tau_1)}\right]d\tau_2d\tau_1
	+q(t)\circ\int_0^tX^{u,\tau}_{v(\tau)}\,d\tau
	\circ\int_0^tX^{u,\tau}_{v(\tau)}\,d\tau .
\end{equation}
The second summand is a genuinely second-order differential operator; it vanishes when
evaluated on $v\in\ker d_0G^t_{q(t)}$, and in that case $d_0^2G^t_{q(t)}(v)$ is a tangent
vector at $q(t)$, defined modulo $\im d_0G^t_{q(t)}$. Consequently, for any adjoint curve
$t\mapsto\lambda(t)$ as in \eqref{eq:covectcurve}, the duality product
$\left\langle\lambda(t),d_0^2G^t_{q(t)}(v)\right\rangle$ defines an intrinsic quadratic
form on $\ker d_0G^t_{q(t)}$.

\begin{remark}[The second-order operator is not discarded]\label{rmk:notdiscarded}
	In Section~\ref{sec:separation} we shall need the expansion of the end-point map for
	arbitrary admissible variations, not only for those in the kernel of
	$d_0G^t_{q(t)}$. The second summand of \eqref{eq:Second_order_differential} does not
	vanish there, and it will be retained and estimated separately: this is the term
	$S^t(\bv)$ of Lemma~\ref{lem:cubicremainder}. Its size is controlled by the
	defect $\left|d_0G^t_{q(t)}(\bv)\right|$, which by Lemma~\ref{lem:defect} is
	small precisely for exact competitors. Overlooking this term is the one place where the
	naive scheme of Subsection~\ref{ssec:onedir} would fail.
\end{remark}

Writing $V(\tau)=\int_0^\tau v(\theta)\,d\theta$ for the primitive of the variation, and
using the identity $X^{u,\tau}_{v(\tau)}=\frac{d}{d\tau}X^{u,\tau}_{V(\tau)}
-\dot X^{u,\tau}_{V(\tau)}$ together with two integrations by parts, the bracket part of
\eqref{eq:Second_order_differential}, paired with $\lambda(t)$, becomes
\begin{equation}\label{eq:sdif}
	\begin{aligned}
		\int_{\Sigma_2(t)}&\left\langle\lambda(t),
		\left[X^{u,\tau_2}_{v(\tau_2)},X^{u,\tau_1}_{v(\tau_1)}\right](q(t))
		\right\rangle d\tau_2d\tau_1\\
		&=\int_0^t\left\langle\lambda(t),
		\left[X^{u,\tau}_{V(\tau)},X^{u,\tau}_{v(\tau)}\right](q(t))\right\rangle d\tau\\
		&+\int_0^t\left\langle\lambda(t),
		\left[\dot X^{u,\tau}_{V(\tau)},X^{u,\tau}_{V(\tau)}\right](q(t))\right\rangle d\tau\\
		&+\int_0^t\left\langle\lambda(t),
		\left[X^{u,t}_{V(t)}+\int_0^{\tau}\dot X^{u,\theta}_{V(\theta)}\,d\theta,\,
		\dot X^{u,\tau}_{V(\tau)}\right](q(t))\right\rangle d\tau .
	\end{aligned}
\end{equation}
The Goh condition \cite{ASBook,ABB20} states that a strictly singular length minimizer
admits at least one nonzero adjoint curve of covectors such that
\begin{equation}\label{eq:Goh_section}
	\left\langle\lambda(t),\left[X_i,X_j\right](q(t))\right\rangle=0
\end{equation}
holds along the curve for every $t\in[0,1]$ and every pair of indices $1\le i<j\le d$.
Since $\langle\lambda(t),[X^{u,\tau}_a,X^{u,\tau}_b](q(t))\rangle
=\langle\lambda(\tau),[X_a,X_b](q(\tau))\rangle$, relation \eqref{eq:Goh_section}
annihilates the first summand of \eqref{eq:sdif}; from now on we always work with an
adjoint curve satisfying \eqref{eq:Goh_section}. The second necessary condition
classically obtained is the Legendre condition: the quadratic form
\begin{equation}\label{eq:Legcond}
	\begin{aligned}
		Q^t:\R^d&\longrightarrow\R,\\
		w&\longmapsto\left\langle\lambda(t),\left[\dot X_w,X_w\right](q(t))\right\rangle
		=\left\langle\lambda(t),\left[[X_{u(t)},X_w],X_w\right](q(t))\right\rangle
	\end{aligned}
\end{equation}
is non-negative on $u(t)^\perp\subset\R^d$ for every Lebesgue point $t$ of $u$ in $[0,1]$
(in particular a.e.\ on $[0,1]$). This can be restated as follows.

\begin{proposition}\label{prop:essmin}
	Suppose that $Q^t\ge0$ on $u(t)^\perp\subset\R^d$ at every Lebesgue point of $u$. Then
	\[
		\mathrm{ess\,inf}_{t\in[0,1]}
		\left(\min_{w\in u(t)^\perp,\,|w|=1}Q^{t}(w)\right)\ge0 .
	\]
\end{proposition}
\begin{proof}
	Let $\xi_{\min}(t):=\min_{w\in u(t)^\perp,|w|=1}Q^{t}(w)$. Since $Q^t$ is homogeneous
	of degree $2$, we have $Q^t\ge0$ on $u(t)^\perp$ if and only if $\xi_{\min}(t)\ge0$. In
	particular $\xi_{\min}(t)\ge0$ for almost every $t\in[0,1]$, that is
	$\mu\left(\{t\in[0,1]\mid\xi_{\min}(t)<0\}\right)=0$; hence $0$ is an essential lower
	bound for $\xi_{\min}$ and, by definition of essential infimum as the supremum of such
	bounds, $\mathrm{ess\,inf}_{t\in[0,1]}\xi_{\min}(t)\ge0$.
\end{proof}

Proposition~\ref{prop:essmin} opens three possibilities for a candidate optimal control
$u$:
\begin{itemize}
	\item[i)] $\mathrm{ess\,inf}_{t\in[0,1]}\xi_{\min}(t)\ge c>0$;
	\item[ii)] $\mathrm{ess\,inf}_{t\in[0,1]}\xi_{\min}(t)=0$, with $\xi_{\min}(t)>0$
	a.e.\ on $[0,1]$;
	\item[iii)] $\mathrm{ess\,inf}_{t\in[0,1]}\xi_{\min}(t)=0$, with $\xi_{\min}(t)=0$ on
	a set of positive measure.
\end{itemize}
In this paper we deal only with the first case: we assume that there exists $c>0$ such
that
\begin{equation}\label{eq:strongLeg}
	Q^{t}(w)=\left\langle\lambda(t),\left[\dot X_w,X_w\right](q(t))\right\rangle
	\ge c\,|w|^2
\end{equation}
for a.e.\ $t\in[0,1]$ and every $w\in u(t)^\perp$.

\begin{remark}\label{rmk:strongLeg}
	Condition \eqref{eq:strongLeg} is commonly referred to as the strong generalized
	Legendre condition. In \cite{AS95,AS96} the inequality is assumed to hold at
	all points $t\in[0,1]$, an assumption which considerably simplifies the proof of
	the smoothness of $u$. Throughout this paper we assume its validity only on the set of
	Lebesgue points of $u$, and we recover the smoothness by a compactness argument along
	the essential graph \eqref{eq:essgraph}. For rank-two distributions,
	\eqref{eq:strongLeg} has been used to prove the local $W^{1,1}$-minimality of nice
	abnormal curves, i.e.\ under \eqref{eq:nice}; see Subsection~\ref{ssec:nice} and
	Remark~\ref{rmk:niceanalogy}.
\end{remark}

\section{Regularity of the control and adapted coordinates}\label{sec:regularity}

\subsection{Algebraic preliminaries}
It is convenient to encode all the brackets of length three appearing in the analysis
into a single trilinear object. For $z=(q,\lambda)\in T^*M$ we set
\begin{equation}\label{eq:trilinear}
	T_z(a,b,c):=\left\langle\lambda,\left[[X_a,X_b],X_c\right](q)\right\rangle,
	\qquad a,b,c\in\R^d ,
\end{equation}
and we associate with $T_z$ the matrix-valued map
\begin{equation}\label{eq:Rmatrix}
	R(q,\lambda,w)_{ij}:=T_z(w,e_i,e_j),
	\qquad w\in\R^d,\quad i,j=1,\dots,d,
\end{equation}
where $\{e_1,\dots,e_d\}$ is the canonical basis of $\R^d$. Writing
$z(t):=(q(t),\lambda(t))$, the Legendre form \eqref{eq:Legcond} reads
\begin{equation}\label{eq:RviaT}
	Q^{t}(w)=T_{z(t)}(u(t),w,w).
\end{equation}
Two structural identities follow directly from the definition. First, antisymmetry in the
first two arguments:
\begin{equation}\label{eq:T1}
	T_z(a,b,c)=-T_z(b,a,c);
	\qquad\text{in particular}\qquad T_z(a,a,c)=0 .
\end{equation}
Second, taking $a=b=w$ in \eqref{eq:T1} we obtain, for every $w\in\R^d$ and every
$j=1,\dots,d$,
\begin{equation}\label{eq:T2}
	\sum_{i=1}^dw_i\,R(q,\lambda,w)_{ij}=T_z(w,w,e_j)=0,
	\qquad\text{that is}\qquad R(q,\lambda,w)^{T}w\equiv0 .
\end{equation}
The differentiated Goh condition,
$\left\langle\lambda(t),\left[X_{u(t)},[X_i,X_j]\right](q(t))\right\rangle=0$, is best
expressed through the linear map
\begin{equation}\label{eq:Bmap}
	B(z):\R^d\longrightarrow\Lambda^2(\R^d)^*,
	\qquad
	\left(B(z)a\right)_{ij}:=T_z(a,e_i,e_j)-T_z(a,e_j,e_i).
\end{equation}
Indeed, by the Jacobi identity
$\left[X_a,[X_b,X_c]\right]=\left[[X_a,X_b],X_c\right]-\left[[X_a,X_c],X_b\right]$, so
that
\begin{equation}\label{eq:Gohviatrilinear}
	\left\langle\lambda,\left[X_a,[X_b,X_c]\right](q)\right\rangle=T_z(a,b,c)-T_z(a,c,b),
\end{equation}
and the vanishing of the left-hand side for all $b,c$ is exactly the statement
$B(z)a=0$, i.e.\ the symmetry of the bilinear form $T_z(a,\cdot,\cdot)$ in its last two
arguments.

We record four facts used in the proof of the regularity of $u$. The first two are purely
algebraic and pointwise in $z$; the last two are compactness statements along the
reference extremal.

\begin{lemma}\label{lem:rank}
	Let $z=(q,\lambda)\in T^*M$, let $w_0\in\ker B(z)$ with $|w_0|=1$, and assume that
	there exists $c>0$ such that
	\[
		T_z(w_0,y,y)\ge c|y|^2\qquad\text{for every }y\in w_0^\perp .
	\]
	Then $\ker B(z)=\Span\{w_0\}$. Moreover $R(q,\lambda,w_0)$ is symmetric and
	$\ker R(q,\lambda,w_0)=\Span\{w_0\}$.
\end{lemma}
\begin{proof}
	Let $w\in\ker B(z)$ with $w\perp w_0$ and $|w|=1$. Since $w\in\ker B(z)$, the bilinear
	form $T_z(w,\cdot,\cdot)$ is symmetric in its last two arguments, whence
	$T_z(w,w_0,w)=T_z(w,w,w_0)=0$, the last equality following from \eqref{eq:T1}. On the
	other hand, again by \eqref{eq:T1}, $T_z(w,w_0,w)=-T_z(w_0,w,w)$. Combining the two
	identities we obtain $T_z(w_0,w,w)=0$, contradicting $T_z(w_0,w,w)\ge c|w|^2=c>0$.
	Therefore $\ker B(z)\cap w_0^\perp=\{0\}$, and since $w_0\in\ker B(z)$ the first claim
	follows.

	As for the second claim, $w_0\in\ker B(z)$ means that $R(q,\lambda,w_0)$ is symmetric,
	and then \eqref{eq:T2} gives $R(q,\lambda,w_0)w_0=R(q,\lambda,w_0)^Tw_0=0$, so that
	$\ker R(q,\lambda,w_0)\supseteq\Span\{w_0\}$. Conversely, any
	$y\in\ker R(q,\lambda,w_0)\cap w_0^\perp$ satisfies
	$0=\langle y,R(q,\lambda,w_0)y\rangle=T_z(w_0,y,y)\ge c|y|^2$, hence $y=0$.
\end{proof}

\begin{remark}[Geometric reading of Lemma~\ref{lem:rank}]\label{rmk:geomrank}
	The operator $B(z)$ measures the failure of $\lambda$ to annihilate the brackets of
	length three of the form $[X_a,[X_b,X_c]]$; the differentiated Goh condition says that
	the control direction lies in its kernel. Lemma~\ref{lem:rank} states that coercivity
	of the Legendre form forces that kernel to be as small as it can possibly be, namely a
	line. Since the line is spanned by $u(t)$, the control direction is determined -- up to a sign, fixed in Lemma~\ref{lem:sign} by the positivity of $T_z$ -- by the
	state-costate pair alone: the extremal is, in this sense, a genuine characteristic
	curve of a smooth field of directions, which is the true reason behind its smoothness.
\end{remark}

\begin{lemma}\label{lem:sign}
	Let $d\ge2$, let $z=(q,\lambda)\in T^*M$ and let $w_0\in S^{d-1}$ satisfy the
	hypotheses of Lemma~\ref{lem:rank}. Then $w_0$ is the unique $w\in S^{d-1}$ such that
	\[
		B(z)w=0\qquad\text{and}\qquad T_z(w,y,y)>0\quad\text{for every }0\neq y\in w^\perp .
	\]
	Moreover the bordered matrix
	\[
		\begin{pmatrix}R(q,\lambda,w_0)&w_0\\ w_0^T&0\end{pmatrix}
	\]
	is invertible.
\end{lemma}
\begin{proof}
	By Lemma~\ref{lem:rank}, any $w$ as in the statement belongs to
	$\Span\{w_0\}\cap S^{d-1}=\{\pm w_0\}$. Since $T_z$ is linear in its first argument,
	$T_z(-w_0,y,y)=-T_z(w_0,y,y)\le-c|y|^2$ for $y\in w_0^\perp=(-w_0)^\perp$, so that
	$-w_0$ violates the positivity requirement, the space $w_0^\perp$ being non-trivial
	because $d\ge2$. Hence $w=w_0$.

	For the invertibility assertion, assume that
	\[
		\begin{pmatrix}R(q,\lambda,w_0)&w_0\\ w_0^T&0\end{pmatrix}
		\begin{pmatrix}y\\ \alpha\end{pmatrix}=\begin{pmatrix}0\\ 0\end{pmatrix}.
	\]
	The second equation gives $y\in w_0^\perp$. Multiplying the first equation
	$R(q,\lambda,w_0)y+\alpha w_0=0$ by $w_0^T$, and using that $R(q,\lambda,w_0)$ is
	symmetric with $R(q,\lambda,w_0)w_0=0$, so that
	$w_0^TR(q,\lambda,w_0)y=\left\langle R(q,\lambda,w_0)w_0,y\right\rangle=0$, we deduce
	$\alpha|w_0|^2=0$, i.e.\ $\alpha=0$. This leaves $R(q,\lambda,w_0)y=0$ with
	$y\in w_0^\perp$, and Lemma~\ref{lem:rank} forces $y=0$.
\end{proof}

The two lemmas above are pointwise statements. In order to use them along the reference
extremal, where the control is a priori only defined almost everywhere, we study their
behaviour under limits. Let $E\subset[0,1]$ denote the full-measure set of Lebesgue
points of $u$ at which the Goh condition \eqref{eq:Goh_section}, its time derivative, and
the strong Legendre condition \eqref{eq:strongLeg} hold, and consider the essential
graph
\begin{equation}\label{eq:essgraph}
	\mathcal G_0:=\left\{\big(q(t),\lambda(t),u(t)\big)\ \middle|\ t\in E\right\}
	\subset T^*M\times S^{d-1},
	\qquad\mathcal G:=\overline{\mathcal G_0}.
\end{equation}

\begin{lemma}\label{lem:essential-graph}
	The set $\mathcal G$ is compact, and every $(q_*,\lambda_*,w_*)\in\mathcal G$ satisfies
	\[
		|w_*|=1,\qquad B(q_*,\lambda_*)w_*=0,\qquad R(q_*,\lambda_*,w_*)w_*=0 .
	\]
\end{lemma}
\begin{proof}
	The curve $t\mapsto(q(t),\lambda(t))$ is continuous, hence its image
	$\Gamma:=\{(q(t),\lambda(t)):t\in[0,1]\}$ is compact. Since $|u(t)|=1$ for $t\in E$, we
	have $\mathcal G_0\subseteq\Gamma\times S^{d-1}$, and therefore $\mathcal G$ is compact,
	being a closed subset of a compact set.

	For every $t\in E$ the differentiated Goh condition gives $B(z(t))u(t)=0$ and hence, by
	\eqref{eq:T2} and the symmetry of $R(q(t),\lambda(t),u(t))$, also
	$R(q(t),\lambda(t),u(t))u(t)=0$. The maps $z\mapsto B(z)$ and $(z,w)\mapsto R(z,w)w$
	are smooth, hence continuous, so that the sets $\{(z,w):B(z)w=0\}$,
	$\{(z,w):R(z,w)w=0\}$ and $T^*M\times S^{d-1}$ are closed. Since $\mathcal G_0$ is
	contained in all of them, so is its closure $\mathcal G$.
\end{proof}

\begin{lemma}\label{lem:uniform-bordered-invertibility}
	For every $(q_*,\lambda_*,w_*)\in\mathcal G$ one has
	\[
		\left\langle y,R(q_*,\lambda_*,w_*)y\right\rangle\ge c|y|^2
		\qquad\text{for every }y\in w_*^\perp .
	\]
	Consequently $\ker B(q_*,\lambda_*)=\Span\{w_*\}$, the bordered matrix
	\[
		\mathcal M(q_*,\lambda_*,w_*):=
		\begin{pmatrix}R(q_*,\lambda_*,w_*)&w_*\\ w_*^T&0\end{pmatrix}
	\]
	is invertible, and
	$\sup_{\zeta\in\mathcal G}\left\|\mathcal M(\zeta)^{-1}\right\|<\infty$.
\end{lemma}
\begin{proof}
	Let $(q(t_n),\lambda(t_n),u(t_n))\to(q_*,\lambda_*,w_*)$ with $t_n\in E$. Fix
	$y\in w_*^\perp$ and set $y_n:=y-\left\langle y,u(t_n)\right\rangle u(t_n)\in
	u(t_n)^\perp$. Since $u(t_n)\to w_*$ and $\langle y,w_*\rangle=0$, we have $y_n\to y$.
	The strong Legendre condition, written through \eqref{eq:RviaT}, gives
	$\left\langle y_n,R(q(t_n),\lambda(t_n),u(t_n))y_n\right\rangle\ge c|y_n|^2$, and
	passing to the limit, by continuity of $R$,
	$\left\langle y,R(q_*,\lambda_*,w_*)y\right\rangle\ge c|y|^2$.

	By Lemma~\ref{lem:essential-graph} we have $B(q_*,\lambda_*)w_*=0$, so that
	Lemma~\ref{lem:rank} applies at $(q_*,\lambda_*)$ with $w_0=w_*$ and yields
	$\ker B(q_*,\lambda_*)=\Span\{w_*\}$ together with
	$\ker R(q_*,\lambda_*,w_*)=\Span\{w_*\}$; Lemma~\ref{lem:sign} then gives the
	invertibility of $\mathcal M(q_*,\lambda_*,w_*)$. Finally, $\zeta\mapsto\mathcal
	M(\zeta)$ is continuous on the compact set $\mathcal G$ and invertible at each of its
	points, hence $\zeta\mapsto\mathcal M(\zeta)^{-1}$ is continuous on $\mathcal G$ and its
	norm attains a finite maximum there.
\end{proof}

\subsection{Regularity of the control}
The regularity of $u$ is now a matter of assembling the previous lemmas. The mechanism is
the following: the differentiated Goh condition is a linear constraint on the
value of the control, the strong Legendre condition forces the corresponding kernel to be
a line and fixes its orientation, and a bordered implicit function theorem turns this
pointwise selection into a smooth dependence on the state and the covector.

\begin{proposition}\label{prop:ucont}
	Assume that the Goh condition \eqref{eq:Goh_section} and the strong Legendre condition
	\eqref{eq:strongLeg} hold, and let $d\ge2$. Then, for every $t_*\in[0,1]$, there exists
	a unique $w_*\in S^{d-1}$ such that $(q(t_*),\lambda(t_*),w_*)\in\mathcal G$.
	Consequently, $u|_E$ admits a unique continuous extension $\bar u:[0,1]\to S^{d-1}$.
\end{proposition}
\begin{proof}
	Let $t\in E$. Applying \eqref{eq:derpairing} with $Y=[X_i,X_j]$ to the identity
	$\langle\lambda(s),[X_i,X_j](q(s))\rangle=0$, valid for every $s$ by
	\eqref{eq:Goh_section}, we obtain
	$\langle\lambda(t),[X_{u(t)},[X_i,X_j]](q(t))\rangle=0$ for all $i,j$, which by
	\eqref{eq:Bmap} and \eqref{eq:Gohviatrilinear} reads
	\begin{equation}\label{eq:linearGoh}
		B(z(t))u(t)=0 .
	\end{equation}
	Together with \eqref{eq:strongLeg} and \eqref{eq:RviaT}, Lemma~\ref{lem:rank} yields
	\begin{equation}\label{eq:kernelline}
		\ker B(z(t))=\ker R(q(t),\lambda(t),u(t))=\Span\{u(t)\},\qquad t\in E.
	\end{equation}
	Fix $t_*$. Since $E$ is dense, choose $t_n\in E$ with $t_n\to t_*$. Up to a
	subsequence, $u(t_n)\to w_*\in S^{d-1}$, and continuity of $z$ gives
	$(z(t_*),w_*)\in\mathcal G$.

	Let $w_*,\widetilde w_*\in S^{d-1}$ be such that $(z(t_*),w_*)$ and
	$(z(t_*),\widetilde w_*)$ belong to $\mathcal G$, and choose sequences
	$\tau_n,\sigma_n\in E$ with $(z(\tau_n),u(\tau_n))\to(z(t_*),w_*)$ and
	$(z(\sigma_n),u(\sigma_n))\to(z(t_*),\widetilde w_*)$. Passing to the limit in
	\eqref{eq:linearGoh} gives $B(z(t_*))w_*=B(z(t_*))\widetilde w_*=0$. For
	$y\in w_*^\perp$, set $y_n=y-\langle y,u(\tau_n)\rangle u(\tau_n)$; then $y_n\to y$ and the
	strong Legendre inequality passes to the limit:
	\begin{equation}\label{eq:limitLegendre}
		T_{z(t_*)}(w_*,y,y)\ge c|y|^2,\qquad y\in w_*^\perp .
	\end{equation}
	Lemma~\ref{lem:rank} gives $\ker B(z(t_*))=\Span\{w_*\}$, hence
	$\widetilde w_*=\pm w_*$. The same limiting inequality for $\widetilde w_*$ excludes
	the negative sign, because $d\ge2$ and $T_z(-w_*,y,y)=-T_z(w_*,y,y)<0$ for non-zero
	$y\in w_*^\perp$. Thus every fiber is a singleton.

	Define $\bar u(t)$ as its unique element; if $t\in E$, then $\bar u(t)=u(t)$. If
	$t_n\to t_*$ and a subsequence $\bar u(t_{n_k})\to w$, closedness of $\mathcal G$
	implies $(z(t_*),w)\in\mathcal G$, hence $w=\bar u(t_*)$. Compactness of the sphere then
	implies convergence of the whole sequence, so $\bar u$ is continuous.
\end{proof}

From now on we identify $u$ with its continuous representative $\bar u$; note that, by
continuity and by Lemma~\ref{lem:uniform-bordered-invertibility}, the Goh and strong
Legendre conditions now hold at every $t\in[0,1]$, and
$\big(q(t),\lambda(t),u(t)\big)\in\mathcal G$ for every $t\in[0,1]$. For future reference
we record the precise form of the quantitative implicit function theorem needed below;
its proof is classical, see \cite{KP2013} and \cite[Theorem~15.1]{Deimling1985}.

\begin{lemma}[Quantitative implicit-function lemma]\label{lem:quantitative-IFT}
	Let $P,Y,Z$ be finite dimensional normed spaces, let $\Phi:P\times Y\to Z$ be $C^2$,
	and let $K\subset P\times Y$ be compact with $\Phi\equiv0$ on $K$. Assume that
	$D_y\Phi$ is invertible at every point of $K$, with
	$\sup_K\|D_y\Phi^{-1}\|\le M$. If $D_y\Phi$ is $L$-Lipschitz and $D_p\Phi$ is bounded
	by $N$ on a compact neighborhood of $K$, then any $r,\rho>0$ satisfying
	\[
		ML(r+\rho)\le\frac12,\qquad MNr\le\frac\rho2 ,
	\]
	and remaining inside that neighborhood, have the following property: for each
	$(p_*,y_*)\in K$ and each $p\in B_r(p_*)$, the equation $\Phi(p,y)=0$ has a unique
	solution in $B_\rho(y_*)$, depending smoothly on $p$.
\end{lemma}

\begin{proposition}\label{prop:smooth}
	Under the Goh condition \eqref{eq:Goh_section} and the strong Legendre condition
	\eqref{eq:strongLeg}, the continuous representative of $u$ is smooth on $[0,1]$.
\end{proposition}
\begin{proof}
	Define
	\begin{equation}\label{eq:Psi}
		\Psi(q,\lambda,w,\mu)=
		\begin{pmatrix}R(q,\lambda,w)w+\mu w\\[1mm]\frac12(|w|^2-1)\end{pmatrix}.
	\end{equation}
	Along the extremal, $\Psi(q(t),\lambda(t),u(t),0)=0$. Writing
	$F(q,\lambda,w)=R(q,\lambda,w)w$, so that $F_i(q,\lambda,w)=T_z(w,e_i,w)$, and
	differentiating in $w$ along a direction $m\in\R^d$ we obtain
	\[
		\big(D_wF(q,\lambda,w)m\big)_i\Big|_{w=u(t)}
		=T_{z(t)}(m,e_i,u(t))+\big(R(q(t),\lambda(t),u(t))m\big)_i .
	\]
	The first summand equals $-\left\langle\lambda(t),[X_{u(t)},[X_m,X_i]](q(t))\right\rangle$
	and vanishes by the differentiated Goh condition \eqref{eq:linearGoh}; therefore
	\begin{equation}\label{eq:jacPsi}
		D_{(w,\mu)}\Psi(q(t),\lambda(t),u(t),0)=
		\begin{pmatrix}R(q(t),\lambda(t),u(t))&u(t)\\ u(t)^T&0\end{pmatrix}
		=\mathcal M(q(t),\lambda(t),u(t)).
	\end{equation}
	By Lemma~\ref{lem:uniform-bordered-invertibility} these matrices are invertible with
	uniformly bounded inverse on the compact essential graph. Applying
	Lemma~\ref{lem:quantitative-IFT} with $p=(q,\lambda)$ and $y=(w,\mu)$ we obtain radii
	$r,\rho>0$, independent of the base point, and unique smooth local branches
	$(w_{\zeta_*},\mu_{\zeta_*})$. Taking the scalar product of the first block of
	$\Psi=0$ with $w$, and using $R(z,w)^Tw=0$ from \eqref{eq:T2} together with $|w|=1$, we
	get $\mu_{\zeta_*}=0$.

	Continuity of $(z(t),u(t))$ and local uniqueness imply, near every $t_0$,
	\begin{equation}\label{eq:localbranch}
		u(t)=w_{t_0}(z(t)),
	\end{equation}
	and on overlaps both branches coincide with the already fixed continuous
	representative. Since the state-costate equations
	\[
		\begin{cases}
			\dot q(s)=X_{u(s)}(q(s)),\\[1mm]
			\dot\lambda(s)=-\,\partial_q\left\langle\lambda(s),X_{u(s)}(q(s))\right\rangle ,
		\end{cases}
	\]
	i.e.\ the lift to $T^*M$ of the flow of $X_{u}$,
	have continuous right-hand side, $z=(q,\lambda)\in C^1$, and \eqref{eq:localbranch}
	gives $u\in C^1$. If $u\in C^k$, the state-costate equations give $z\in C^{k+1}$ and
	the smooth branch gives $u\in C^{k+1}$. Induction proves smoothness.
\end{proof}

\begin{remark}[Higher rank versus niceness]\label{rmk:niceanalogy}
	Proposition~\ref{prop:smooth} is the higher-rank counterpart of the regularity
	statements available for nice abnormal extremals of rank-two distributions
	\cite{AS95},\cite[Chapter~12]{ABB20}. As explained in Subsection~\ref{ssec:nice}, for
	$d=2$ the niceness condition \eqref{eq:nice} is equivalent to $B(z)\neq0$ and hence,
	$\Lambda^2(\R^2)^*$ being one-dimensional, to \eqref{eq:onedim}. For $d\ge3$ the two
	requirements part company: $B(z)\neq0$ is one scalar condition, whereas
	\eqref{eq:onedim} is the maximality $\rank B(z)=d-1$. What the present section shows is
	that the strong Legendre condition delivers \eqref{eq:onedim} for free, and with it the
	whole regularity theory, without any assumption on the flag $\Delta\subset\Delta^2
	\subset\Delta^3$.
\end{remark}

\subsection{The adapted frame and the adapted chart}\label{ssec:adapted}
Now that $u$ is known to be smooth, we normalize the geometric data once and for all. The
following lemma collects, in a single statement, the choice of the horizontal frame, of
the chart, of the dual coframe and of the uniform bounds; every constant appearing in the
rest of the paper is a supremum taken over the compact sets produced here, and is
therefore independent of $t$.

\begin{lemma}[Adapted normalization]\label{lem:adapted}
	There exist $t_0\in(0,1]$, open sets $\mathcal O_1\subset\subset\mathcal O\subset M$
	containing the arc $\Gamma_{t_0}:=q([0,t_0])$, a chart
	$\kappa:\mathcal O\to\Omega\subset\R^n$, and a smooth orthonormal frame
	$X_1,\dots,X_d$ of $\Delta$ on $\mathcal O$, with the following properties.
	\begin{itemize}
		\item[(i)] (Straightening)
		$X_1(q(s))=\dot q(s)$ and $u(s)=e_1$ for every $s\in[0,t_0]$.
		\item[(ii)] (Coframe) There are covectors
		$\eta_1(t),\dots,\eta_d(t)\in T^*_{q(t)}M$, depending continuously on
		$t\in[0,t_0]$, with
		\begin{equation}\label{eq:coframe}
			\left\langle\eta_k(t),X_j(q(t))\right\rangle=\delta_{kj},
			\qquad j,k=1,\dots,d,
			\qquad
			C_\eta:=\sup_{t\in[0,t_0]}\max_{1\le i\le d}|\eta_i(t)|<\infty .
		\end{equation}
		\item[(iii)] (Affine abnormal coordinate) Setting
		$\kappa_t:=\kappa-\kappa(q(t))$ and
		$\ell_t:=\lambda(t)\circ d_{\kappa(q(t))}\kappa^{-1}\in(\R^n)^*$, the function
		\begin{equation}\label{eq:alinear}
			a_t:\mathcal O\to\R,\qquad a_t(x):=\left\langle\ell_t,\kappa_t(x)\right\rangle
		\end{equation}
		satisfies $a_t(q(t))=0$, $d_{q(t)}a_t=\lambda(t)$ and
		$D^k_y\left(a_t\circ\kappa_t^{-1}\right)\equiv0$ for every $k\ge2$; moreover
		$C_\ell:=\sup_{t\in[0,t_0]}|\ell_t|<\infty$.
		\item[(iv)] (Uniform bounds) Let $\mathcal B:=\kappa(\overline{\mathcal
		O_1})$, a compact subset of $\Omega$. Every supremum of the coordinate expressions
		of the vector fields $X_i$, of their transported versions
		$X^\tau_i:=e^{(t-\tau)X_1}_*X_i$, and of all their partial derivatives with respect
		to $x$ and $\tau$ of total order at most three, taken over $\mathcal B$ and over
		$\tau,t\in[0,t_0]$, is finite.
		\item[(v)] (Localization) Set
		\[
			L_0:=\sqrt d\,\max_{1\le i\le d}\sup_{\tau,t\in[0,t_0],\,x\in\mathcal B}
			\left|X_i^\tau(x)\right|,
			\qquad
			\delta:=\inf_{s\in[0,t_0]}\operatorname{dist}\big(\kappa(q(s)),
			\R^n\setminus\kappa(\mathcal O_1)\big).
		\]
		Then $\delta>0$ and $2L_0t_0<\delta$.
	\end{itemize}
\end{lemma}
\begin{proof}
	Since $|u(t)|=1$ and $X_1,\dots,X_d$ are linearly independent,
	$\dot q(t)=X_{u(t)}(q(t))\neq 0$; after decreasing $t_0>0$, the arc $q|_{[0,t_0]}$ is a
	smooth embedded curve, contained in a contractible coordinate neighborhood
	$\mathcal O$ with chart $\kappa$. Extend the coefficients of $u$ off the arc by
	$\widetilde u_i(x):=u(s)$ if $x$ lies on the slice of a tubular neighborhood through
	$q(s)$, put $\widetilde Y:=\sum_i\widetilde u_iX_i$ and, after shrinking $\mathcal O$
	so that $\widetilde Y\neq0$,
	\[
		Y_1:=\frac{\widetilde Y}{\|\widetilde Y\|_g} ,
	\]
	a smooth unit horizontal field with $Y_1(q(s))=\dot q(s)$. Completing $Y_1$ to a smooth
	orthonormal frame $Y_1,\dots,Y_d$ of $\Delta$ on the contractible set $\mathcal O$ and
	relabelling $Y_i$ as $X_i$ gives (i). Since $X_1(q(t)),\dots,X_d(q(t))$ are linearly
	independent, the covectors in (ii) exist; they are not unique when $d<n$, and we fix
	once and for all the choice obtained by Gram--Schmidt orthonormalization of
	$X_1(q(t)),\dots,X_d(q(t))$ with respect to an auxiliary Riemannian metric on $M$,
	declaring $\eta_k(t)$ to vanish on the corresponding orthogonal complement; this choice
	is continuous in $t$, whence $C_\eta<\infty$ by compactness of $[0,t_0]$. Item (iii)
	is immediate: $\kappa_t$ is a translation of $\kappa$, so $d\kappa_t=d\kappa$ and
	$a_t\circ\kappa_t^{-1}=\left\langle\ell_t,\cdot\right\rangle$ is linear; continuity of
	$t\mapsto\ell_t$ gives $C_\ell<\infty$. Finally, choose $\mathcal O_1$ with
	$\Gamma_{t_0}\subset\mathcal O_1\subset\subset\mathcal O$ and decrease $t_0$ once more
	so that all the flow segments occurring below stay in $\overline{\mathcal O_1}$; all
	the data in (iv) are smooth functions on the compact set
	$\mathcal B\times[0,t_0]\times[0,t_0]$, hence bounded there. As for (v),
	$\delta>0$ because $s\mapsto\kappa(q(s))$ is continuous on a compact interval with
	values in the open set $\kappa(\mathcal O_1)$; moreover, when $t_0$ decreases $L_0$ does
	not increase and $\delta$ does not decrease, so that $2L_0t_0<\delta$ is obtained by
	decreasing $t_0$ once more.
\end{proof}

\begin{remark}[Change of frame]\label{rmk:frame}
	The frame of Lemma~\ref{lem:adapted} differs from the one used so far by a smooth
	$O(d)$-valued function $\mathcal R$ on $\mathcal O$: a horizontal curve $\gamma$ with
	control $\tilde u$ in the old frame has control $\mathcal R(\gamma)^T\tilde u$ in the
	new one. This correspondence is a smooth change of variables between neighborhoods of
	$u$ in $L^2$ which preserves pointwise the norm of the control, hence the length, the
	energy and the end-point. Consequently the image of the differential of the end-point
	map, and the covectors orthogonal to it, do not depend on the frame. That the Goh
	condition and the Legendre form are not affected either is the content of
	Lemma~\ref{lem:frameinv} below, since the brackets of the new fields contain
	derivatives of the coefficients of $\mathcal R$.
\end{remark}
\begin{lemma}[Invariance under change of frame]\label{lem:frameinv}
	Let $X_1,\dots,X_d$ be the frame used so far, let $Y_a=\sum_{i}\mathcal R_{ia}X_i$,
	$a=1,\dots,d$, be the frame of Lemma~\ref{lem:adapted}, and let $q\in\mathcal O$,
	$\lambda\in T^*_qM$ with $\lambda\perp\Delta_q$ and $\lambda\perp\Delta^2_q$. Then,
	writing $\mathcal R=\mathcal R(q)$,
	\[
		\langle\lambda,[Y_a,Y_b](q)\rangle
		=\sum_{i,j}\mathcal R_{ia}\mathcal R_{jb}\,\langle\lambda,[X_i,X_j](q)\rangle,
	\]
	\[
		\langle\lambda,[[Y_a,Y_b],Y_c](q)\rangle
		=\sum_{i,j,k}\mathcal R_{ia}\mathcal R_{jb}\mathcal R_{kc}\,
		\langle\lambda,[[X_i,X_j],X_k](q)\rangle .
	\]
	In particular, \eqref{eq:Goh_section} and \eqref{eq:strongLeg} hold in the frame
	$Y_1,\dots,Y_d$ with the same constant $c$.
\end{lemma}
\begin{proof}
	Let $A,B,C$ be smooth sections of $\Delta$ and $f$ a smooth function. Then
	\[
		[fA,B]=f[A,B]-(Bf)A,\qquad
		[[A,B],fC]=f[[A,B],C]+([A,B]f)\,C,
	\]
	\[
		[[fA,B],C]=f[[A,B],C]-(Cf)[A,B]-(Bf)[A,C]+(C(Bf))\,A ,
	\]
	and the case of $fB$ follows from $[[A,fB],C]=-[[fB,A],C]$. Since $\lambda$ annihilates
	$\Delta_q$ and $[A,B](q),[A,C](q)\in\Delta^2_q$, all the correction terms vanish after
	pairing with $\lambda$. Hence $\langle\lambda,[A,B](q)\rangle$ and
	$\langle\lambda,[[A,B],C](q)\rangle$ are multilinear over $C^\infty$ functions in the
	sections of $\Delta$, i.e.\ they depend only on the values of $A,B,C$ at $q$, and the
	two formulas follow by expanding $Y_a=\sum_i\mathcal R_{ia}X_i$. By
	\eqref{eq:Goh_section} and the Pontryagin Maximum Principle the hypothesis on
	$\lambda=\lambda(t)$ holds for every $t$, so the first formula gives the Goh condition
	for the $Y$'s. Moreover $Y_1(q(t))=\dot q(t)=X_{u(t)}(q(t))$ gives
	$\mathcal R(q(t))e_1=u(t)$, so that, by the second formula and the orthogonality of
	$\mathcal R$, for $w\perp e_1$,
	$\langle\lambda(t),[[Y_1,Y_w],Y_w](q(t))\rangle
	=Q^t\big(\mathcal R(q(t))w\big)\ge c\,|\mathcal R(q(t))w|^2=c|w|^2$, because
	$\mathcal R(q(t))w\perp u(t)$.
\end{proof}
From now on $X_1,\dots,X_d$ denotes the frame of Lemma~\ref{lem:adapted}.

\begin{remark}[How the chart is used]\label{rmk:chart}
	Three features of Lemma~\ref{lem:adapted} are used repeatedly, and it is worth
	isolating them once, so as not to return to the matter again.
	\begin{itemize}
		\item[(a)] The parameter $t$ only moves the origin:
		$\kappa_t=\kappa-\kappa(q(t))$. Since translations do not affect derivatives, all
		coordinate derivatives may be computed in the single fixed chart $\kappa$, and all
		the resulting bounds are uniform in $t$. In the coordinates $\kappa_t$ the set
		$\mathcal B$ is translated by $-\kappa(q(t))$ and, by Lemma~\ref{lem:adapted}(v), the
		open ball of radius $\delta$ centred at the origin is contained in the translated
		image of $\mathcal O_1$, hence in the translated set $\mathcal B$. From now on we identify points of
		$\mathcal O$ with their coordinates through $\kappa_t$; in particular
		$q(t)$ becomes the origin of $\R^n$, $G^t_{q(t)}(\bv)$ is regarded as a vector of
		$\R^n$, and $a_t(x)=\left\langle\ell_t,x\right\rangle$.
		\item[(b)]The abnormal coordinate $a_t$ is affine, so its derivatives of order $\ge2$ vanish identically.
		Every nonlinear term in the expansions below comes from the coordinate expressions
		of the transported fields and of their flows, never from $a_t$.
		\item[(c)] Every constant denoted by $C$, $K$, $\Lambda$, $L$ or
		$\delta$, with or without subscripts, is a supremum of the type described in
		Lemma~\ref{lem:adapted}(iv); it depends only on the reference extremal, on the
		adapted frame and coframe, and on $t_0$, and never on $t\in(0,t_0]$ nor on the
		variation.
	\end{itemize}
\end{remark}

\begin{remark}[Geometry of the adapted normalization]\label{rmk:geomstraight}
No equiregularity assumption is used in this construction. The frame is adapted only to the reference direction $u$, and not to the flag $\Delta \subseteq \Delta^2 \subseteq \cdots$, while $\kappa$ is an arbitrary smooth chart, rather than a system of privileged coordinates.
	
    Item (i) of Lemma~\ref{lem:adapted} makes the reference control constant, equal
	to $e_1$. All the geometry is thereby transferred to the frame, and the Goh condition
	acquires a transparent meaning: $\lambda(t)$ annihilates $\Delta_{q(t)}$ by the PMP and
	$\Delta^2_{q(t)}$ by \eqref{eq:Goh_section}, so the abnormal codirection is invisible to
	first and second order along $\Delta$, and the leading term it sees is the
	length-three bracket $\left[[X_{e_1},X_w],X_w\right]$ --- precisely the Legendre form.
	Coercivity of $Q^t$ is thus a uniform transversality requirement of the arc with respect
	to the flag $\Delta\subset\Delta^2\subset\Delta^3$, in all the $d-1$ directions
	orthogonal to the motion.
\end{remark}

From now on we systematically split every vector of $\R^d=\R\oplus\R^{d-1}$ into its
first component and the remaining ones, denoting the full split object by a boldface
letter:
\begin{equation}\label{eq:boldnotation}
	\bw=(w_1,w)\in\R\oplus\R^{d-1},
	\quad\bv=(v_1,v),
	\quad\bV(\tau)=\big(V_1(\tau),V(\tau)\big):=\int_0^\tau\bv(\theta)\,d\theta ,
\end{equation}
so that $\bv\in L^2([0,t],\R^d)$ is an admissible variation of the reference control,
$v_1$ and $V_1$ are its first component and the corresponding primitive, and
$v=(v_2,\dots,v_d)$, $V=(V_2,\dots,V_d)$ collect the components transverse to the
reference direction. The reference control itself is $u=(1,0)$. Plain letters always
denote the transverse $(d-1)$-dimensional part, boldface letters the full
$d$-dimensional object. With this choice of frame the strong Legendre condition becomes
\begin{equation}\label{eq:slg1}
	Q^t(w)=\left\langle\lambda(t),\left[\dot X_w,X_w\right](q(t))\right\rangle\ge c|w|^2,
	\qquad w\in\R^{d-1},\quad t\in[0,t_0],
\end{equation}
where we used the identification $u(t)^\perp=\{0\}\oplus\R^{d-1}$, valid for every $t$.
Notice that only at this point are we entitled to pass from an a.e.\ statement to one
holding for every $t$, since we have shown that $u$ is smooth.

\subsection{Reduction to the admissible set}
Our objective is to prove that there exists $t>0$ such that $u|_{[0,t]}$ is a minimizer
of the length functional
\begin{equation}\label{eq:length}
	\Len:L^2([0,t],\R^d)\longrightarrow\R,
	\qquad
	\Len(\tilde u)=\int_0^t|\tilde u(\tau)|\,d\tau
\end{equation}
within the set of admissible curves $\gamma_{\tilde u}$ with
$F_q^t(\tilde u)=F_q^t(u)=q(t)$, and that its minimizers are exactly the
reparametrizations of $u|_{[0,t]}$. Fix $t\in(0,t_0]$, set $u^t:=u|_{[0,t]}$ and let
$\mathcal U\subset L^2([0,t],\R^d)$ be the set of controls on which $F^t_q$ is well
defined. Since we are interested in minimizers, it is not restrictive to consider only
competitors with $\Len(\tilde u)\le\Len(u^t)=t$. Moreover $\Len(\tilde u)>0$ for every
competitor: otherwise $\tilde u=0$ almost everywhere, $\gamma_{\tilde u}$ is constant and
$F^t_q(\tilde u)=q\neq q(t)$, because the arc $q|_{[0,t_0]}$ is embedded.

Let $\varphi:[0,t]\to[0,t]$ be the nondecreasing absolutely continuous surjection
\begin{equation}\label{eq:timepar}
	\varphi(\tau):=\frac{t}{\Len(\tilde u)}\int_0^\tau|\tilde u(\theta)|\,d\theta ,
\end{equation}
and let $\varphi^{\#}(s):=\min\{\tau\in[0,t]:\varphi(\tau)=s\}$ be its generalized
inverse, which is strictly increasing and satisfies $\varphi\circ\varphi^\#=\mathrm{id}$.
\begin{lemma}[Constant-speed reparametrization]\label{lem:reparam}
	Let $\gamma_{\tilde u}$ be an admissible curve on $[0,t]$ with $\Len(\tilde u)>0$, and
	set $\eta:=\gamma_{\tilde u}\circ\varphi^\#$. Then $\gamma_{\tilde u}=\eta\circ\varphi$,
	$\eta$ is an admissible curve with the same end-points as $\gamma_{\tilde u}$, and its
	control $\tilde u_\varphi$ satisfies $|\tilde u_\varphi|=\Len(\tilde u)/t$ a.e.\ on
	$[0,t]$.
\end{lemma}
\begin{proof}
	On every interval where $\varphi$ is constant one has $\tilde u=0$ a.e., hence
	$\gamma_{\tilde u}$ is constant there. Since $\varphi$ is constant on
	$[\varphi^\#(\varphi(\tau)),\tau]$, we get
	$\gamma_{\tilde u}(\tau)=\gamma_{\tilde u}(\varphi^\#(\varphi(\tau)))=\eta(\varphi(\tau))$,
	and evaluating at $\tau=0,t$ gives the claim on the end-points. Let $d_{SR}$ be the
	sub-Riemannian distance. For $0\le s_1\le s_2\le t$ and $a_i:=\varphi^\#(s_i)$ we have
	$a_1\le a_2$ and
	\[
		d_{SR}\big(\eta(s_1),\eta(s_2)\big)\le\ell\big(\gamma_{\tilde u}|_{[a_1,a_2]}\big)
		=\int_{a_1}^{a_2}|\tilde u|=\frac{\Len(\tilde u)}{t}\big(\varphi(a_2)-\varphi(a_1)\big)
		=\frac{\Len(\tilde u)}{t}(s_2-s_1).
	\]
	Thus $\eta$ is Lipschitz with respect to $d_{SR}$; hence it is admissible, and the norm
	of its control is its metric speed, so that $|\tilde u_\varphi|\le\Len(\tilde u)/t$ a.e.\
	(see e.g.\ \cite{ABB20}). Finally, $\eta=\gamma_{\tilde u}\circ\varphi^\#$ and
	$\gamma_{\tilde u}=\eta\circ\varphi$ are compositions with monotone maps, under which the
	length can only decrease; hence $\ell(\eta)=\ell(\gamma_{\tilde u})=\Len(\tilde u)$,
	i.e.\ $\int_0^t|\tilde u_\varphi|=\Len(\tilde u)$. Together with
	$|\tilde u_\varphi|\le\Len(\tilde u)/t$ on an interval of length $t$, this forces
	equality a.e.
\end{proof}
The reduction to curves with a.e.\ constant speed allows us to restrict the search for
competitors to controls $\bv=(v_1,v)$ such that
\begin{equation}\label{eq:conepointwise}
	\big(1+v_1(\tau)\big)^2+|v(\tau)|^2\le1,
	\qquad\text{a.e. }\tau\in[0,t],
\end{equation}
whence in particular
\be\label{eq:cone}
	v_1(\tau)\le0,\qquad\text{a.e. }\tau\in[0,t].
\ee

\begin{remark}[The admissible set]\label{rmk:cone}
	Constraint \eqref{eq:conepointwise} says that $u+\bv$ takes values in the closed unit
	ball, i.e.\ that the competitor is no longer than the reference arc at every instant;
	geometrically, $\bv(\tau)$ ranges in the closed unit ball centred at $-u$, which is
	tangent at the origin to the hyperplane $\{v_1=0\}$. Its linearization at the origin is
	the half-space \eqref{eq:cone}. The set of admissible variations is therefore convex
	and bounded in $L^\infty$, but it is neither a linear subspace nor a cone, since it is
	not invariant under multiplication by positive scalars: this is what makes all the
	estimates below uniform, rather than asymptotic as $\bv\to0$.
\end{remark}

Finally, in these coordinates the time-$t$ end-point map becomes
\be\label{eq:SLCEndpoint}
	\begin{aligned}
		F^t_{q}(u+\bv)&=q\circ\eexp\int_0^t\big(1+v_1(\tau)\big)X_1+X_{v(\tau)}\,d\tau\\
		&=q(t)\circ\eexp\int_0^tv_1(\tau)X_1+X^\tau_{v(\tau)}\,d\tau=G_{q(t)}^t(\bv),
	\end{aligned}
\ee
with $X^\tau:=e^{(t-\tau)X_1}_*X$; equivalently
$X^\tau_{\bv(\tau)}=\sum_{i=1}^dv_i(\tau)X^\tau_i$ with $X^\tau_1\equiv X_1$. Since
$\frac{d}{ds}(e^{sX_1})_*X=-[X_1,(e^{sX_1})_*X]$ and $s=t-\tau$, we have
$\dot X^\tau=[X_1,X^\tau]$, the relation used throughout. For future
purposes we record the following structural property of the adapted frame:
\begin{equation}\label{eq:X1_identity}
	\dot X_1^\tau=[X_1,X_1^\tau]=[X_1,X_1]\equiv0 .
\end{equation}

\section{The reduced Hessian and its coercivity}\label{sec:hessian}

\subsection{Reduction of the Hessian}
The aim of this section is to prove that the projected Hessian
$\left\langle\lambda(t),d_0^2G_{q(t)}^t(\bv)\right\rangle$ is coercive, on the admissible
set and for short times, with respect to a norm strong enough to control the nonlinear
remainders. As observed in \eqref{eq:X1_identity}, $\dot X_1^\tau=0$, so that the first
component drops out of the derivative:
\[
	\dot X^\tau_{\bV(\tau)}=\sum_{j=2}^dV_j(\tau)\dot X_j^\tau
	=\sum_{j=2}^dV_j(\tau)[X_1,X_j^\tau]=\dot X^\tau_{V(\tau)} .
\]
Moreover, the relation
$\left\langle\lambda(t),X_j^\tau(q(t))\right\rangle\equiv0$ holds for every
$\tau\in[0,t]$ by abnormality, and can therefore be differentiated twice in $\tau$. Since
$\dot X_j^\tau=[X_1,X_j^\tau]$, we find
\begin{equation}\label{eq:cruc_vanishing}
	\frac{d^2}{d\tau^2}\left\langle\lambda(t),X_j^\tau(q(t))\right\rangle
	=\left\langle\lambda(t),[X_1,\dot X_j^\tau](q(t))\right\rangle=0 ,
\end{equation}
where the last equality holds because the function is identically zero. Its first
derivative equals $\langle\lambda(t),\dot X_j^\tau(q(t))\rangle
=\langle\lambda(\tau),[X_1,X_j](q(\tau))\rangle$ by the transport argument of
Lemma~\ref{lem:transport}, and vanishes by \eqref{eq:Goh_section}.
Hence every occurrence of $V_1$ disappears from \eqref{eq:sdif}: its first summand is
annihilated by the Goh condition, while in the other two $V_1$ can only appear through
$X_1=X^\tau_1$ inside a bracket with some $\dot X_j^\tau$ (recall $\dot X_1^\tau=0$), which
vanishes by \eqref{eq:cruc_vanishing}. The bracket part of the projected Hessian
reduces to
\begin{equation}\label{eq:hessGnoV1}
	\begin{aligned}
		\mathcal Q^t(\bv)&:=\int_{\Sigma_2(t)}\left\langle\lambda(t),
		\left[X^{\tau_2}_{\bv(\tau_2)},X^{\tau_1}_{\bv(\tau_1)}\right](q(t))
		\right\rangle d\tau_2d\tau_1\\
		&=\int_0^t\left\langle\lambda(t),
		\left[\dot X^{\tau}_{V(\tau)},X^{\tau}_{V(\tau)}\right](q(t))\right\rangle d\tau\\
		&+\int_0^t\left\langle\lambda(t),
		\left[X_{V(t)}+\int_0^{\tau}\dot X^{\theta}_{V(\theta)}\,d\theta,\,
		\dot X^{\tau}_{V(\tau)}\right](q(t))\right\rangle d\tau .
	\end{aligned}
\end{equation}
Two issues arise from \eqref{eq:hessGnoV1}. First, the right-hand side depends on the
variation only through the pair
$\left(V(t),V\right)\in\R^{d-1}\oplus L^2([0,t],\R^{d-1})$: what the strong Legendre
condition can deliver is therefore a coercivity for the $L^2$-norm of the
primitives, a property in general too weak to control the full expansion of
$G^t_{q(t)}(\bv)$ in terms of the $L^2$-norm of $\bv$. Second, the contribution of the
first coordinate --- completely invisible to the Hessian, but relevant at higher orders
--- has to be recovered by other means. The results below are the analytic tools that
overcome both difficulties.

Fix on $H^1([0,t],\R^d)$ the scalar product
\begin{equation}\label{eq:scalprodendpointL2}
	\left\langle x,y\right\rangle_{\mathcal H}
	=\left\langle x(t),y(t)\right\rangle_{\R^d}
	+\int_0^t\left\langle x(\tau),y(\tau)\right\rangle_{\R^d}d\tau
\end{equation}
and the corresponding induced norm
\begin{equation}\label{eq:Hnorm}
	\|x\|_{\mathcal H}^2=|x(t)|^2+\|x\|_{L^2([0,t],\R^d)}^2 .
\end{equation}

\begin{remark}[What $\|\cdot\|_{\mathcal H}$ measures]\label{rmk:Hnorm}
	The norm \eqref{eq:Hnorm} is evaluated on the primitive $\bV$ of the variation,
	not on the variation itself; it is therefore a weak, $H^{-1}$-type norm on $\bv$, and
	it is exactly the norm in which the reduced Hessian \eqref{eq:hessGnoV1} is bounded.
	The boundary term $|x(t)|^2$ cannot be dispensed with: \eqref{eq:hessGnoV1} genuinely
	depends on $V(t)$, and this is the analytic trace of the fact that the end-point is
	fixed.
\end{remark}

\begin{proposition}\label{prop:completion}
	The completion of $H^1([0,t],\R^d)$ with respect to $\|\cdot\|_{\mathcal H}$ is
	isometric to $\mathcal K:=\R^d\oplus L^2([0,t],\R^d)$.
\end{proposition}
\begin{proof}
	Equip $\mathcal K$ with the product Hilbert norm
	$\|(a,f)\|_{\mathcal K}^2=|a|^2+\|f\|_{L^2}^2$. The linear map
	\[
		\Phi:H^1([0,t],\R^d)\longrightarrow\mathcal K,\qquad\Phi(x)=(x(t),x),
	\]
	is by construction an isometry from
	$\big(H^1([0,t],\R^d),\|\cdot\|_{\mathcal H}\big)$ into the complete Hilbert space
	$\mathcal K$; it therefore suffices to prove that its image is dense. Let
	$(a,f)\in\mathcal K$. Since $C_c^\infty\big((0,t),\R^d\big)$ is dense in
	$L^2([0,t],\R^d)$, choose $g_n\in C_c^\infty\big((0,t),\R^d\big)$ with
	$\|g_n-f\|_{L^2}\to0$; compactness of the supports guarantees $g_n(t)=0$ for every $n$.
	For $n>1/t$ pick a cut-off $\psi_n\in C^\infty([0,t])$ with $\psi_n(t)=1$,
	$\psi_n(\tau)=0$ for $\tau\le t-\frac1n$ and $0\le\psi_n\le1$, and set
	$x_n(\tau):=g_n(\tau)+\psi_n(\tau)a$. Then
	$x_n\in C^\infty([0,t],\R^d)\subset H^1([0,t],\R^d)$ and $x_n(t)=a$, so that
	$|x_n(t)-a|=0$ for every $n$. As for the $L^2$-component, the triangle inequality gives
	$\|x_n-f\|_{L^2}\le\|g_n-f\|_{L^2}+\|\psi_na\|_{L^2}$, where the first term vanishes in
	the limit while
	\[
		\|\psi_na\|_{L^2}^2=|a|^2\int_{t-1/n}^t|\psi_n(\tau)|^2d\tau
		\le\frac{|a|^2}{n}\xrightarrow[n\to\infty]{}0 .
	\]
	Therefore $\|\Phi(x_n)-(a,f)\|_{\mathcal K}\to0$, and since $\Phi$ is an isometry onto
	a dense subspace of the complete space $\mathcal K$, the claim follows.
\end{proof}

By \eqref{eq:cruc_vanishing} the quadratic form \eqref{eq:hessGnoV1} does not depend on
the first component of the variation. Denoting by $a=(a_2,\dots,a_d)\in\R^{d-1}$ and
$f=(f_2,\dots,f_d)\in L^2([0,t],\R^{d-1})$ the last $d-1$ components of an element of
$\mathcal K$, we set
\begin{equation}\label{eq:extensHessian}
	\begin{aligned}
		\widehat{H^t}(a,f)
		&:=\int_0^t\left[\dot X^{\tau}_{f(\tau)},X^{\tau}_{f(\tau)}\right](q(t))\,d\tau\\
		&+\int_0^t\left[X_a+\int_0^{\tau}\dot X^{\theta}_{f(\theta)}\,d\theta,\,
		\dot X^{\tau}_{f(\tau)}\right](q(t))\,d\tau\ \in\ T_{q(t)}M ,
	\end{aligned}
\end{equation}
and we write $\lambda(t)\widehat{H^t}(a,f)$ for its scalar projection along the abnormal
covector, so that $\mathcal Q^t(\bv)=\lambda(t)\widehat{H^t}\big(V(t),V\big)$.

\begin{lemma}\label{lem:boundedextension}
	The quadratic mapping $\widehat{H^t}$ is bounded on $\mathcal K$: there exists $K>0$,
	independent of $t\in(0,t_0]$, such that
	\[
		\left|\lambda(t)\widehat{H^t}(a,f)\right|
		\le K\left(|a|^2+\|f\|_{L^2([0,t],\R^{d-1})}^2\right).
	\]
	Consequently, \eqref{eq:hessGnoV1} extends uniquely to a continuous quadratic form on
	$\mathcal K$, which coincides with \eqref{eq:extensHessian}.
\end{lemma}
\begin{proof}
	All the vector fields $X^\tau_j$, $\dot X^\tau_j$ and their brackets, evaluated at
	$q(t)$ and paired with $\lambda(t)$, are continuous functions of
	$(t,\tau)\in[0,t_0]^2$, hence the bilinear forms on $\R^{d-1}\times\R^{d-1}$ defined by them have operator norms
	uniformly bounded by a constant $K_0>0$ by Lemma~\ref{lem:adapted}(iv). Estimating the three summands separately and using
	$\|f\|_{L^1}\le\sqrt t\,\|f\|_{L^2}\le\|f\|_{L^2}$, valid for $t\le1$, we obtain
	\[
		\left|\lambda(t)\widehat{H^t}(a,f)\right|
		\le K_0\|f\|_{L^2}^2+K_0|a|\,\|f\|_{L^1}+K_0\|f\|_{L^1}^2
		\le K\left(|a|^2+\|f\|_{L^2}^2\right)
	\]
	with $K:=\frac52K_0$, where we used
	$|a|\,\|f\|_{L^2}\le\frac12(|a|^2+\|f\|_{L^2}^2)$. The extension statement follows
	from Proposition~\ref{prop:completion} and the density of $H^1([0,t],\R^d)$ in
	$\mathcal K$.
\end{proof}

\subsection{The admissible set controls the strong norm}
\begin{lemma}\label{lem:conenorm}
	Let $\bv=(v_1,v)\in L^2([0,t],\R^d)$ satisfy \eqref{eq:conepointwise}. Then
	\begin{equation}\label{eq:pointwisecone}
		\left|\bv(\tau)\right|^2\le2\left|v_1(\tau)\right|\le4,
		\qquad\text{a.e. }\tau\in[0,t],
	\end{equation}
	and consequently
	\begin{equation}\label{eq:conenorm}
		\left\|\bv\right\|_{L^2([0,t],\R^d)}^2\le2\left|V_1(t)\right|
		\le2\left\|\bV\right\|_{\mathcal H},
		\qquad
		\left\|\bv\right\|_{L^1([0,t],\R^d)}\le2t .
	\end{equation}
\end{lemma}
\begin{proof}
	Expanding the constraint gives $|v(\tau)|^2\le-2v_1(\tau)-v_1(\tau)^2$, whence
	\[
		\left|\bv(\tau)\right|^2=v_1(\tau)^2+\left|v(\tau)\right|^2
		\le-2v_1(\tau)=2\left|v_1(\tau)\right| ,
	\]
	where we used \eqref{eq:cone}. Moreover $\big(1+v_1(\tau)\big)^2\le1$ forces
	$-2\le v_1(\tau)\le0$, which gives the second inequality in \eqref{eq:pointwisecone}
	and, in turn, $\|\bv\|_{L^1}\le2t$. Integrating \eqref{eq:pointwisecone} over $[0,t]$
	and using $V_1(0)=0$ together with $v_1\le0$ we obtain
	\[
		\left\|\bv\right\|_{L^2}^2\le2\int_0^t\left|v_1(\tau)\right|d\tau
		=-2\int_0^tv_1(\tau)\,d\tau=-2V_1(t)=2\left|V_1(t)\right| ,
	\]
	and $|V_1(t)|\le\|\bV\|_{\mathcal H}$ by the very definition of
	$\|\cdot\|_{\mathcal H}$.
\end{proof}

\begin{remark}\label{rmk:nonhomog}
	Estimate \eqref{eq:conenorm} is the structural fact on which the whole remainder
	analysis rests: on the admissible set the square of the strong $L^2$-norm of the
	control is dominated by the first power of the weak norm of its primitive. The
	inequality is of course not homogeneous -- the admissible set is not a linear subspace -- and
	this is precisely what renders a cubic remainder quadratic in $\|\cdot\|_{\mathcal H}$,
	with a small prefactor, rather than merely of order $3/2$. Geometrically, the admissible set
	forces a variation to ``pay'' in the direction $-e_1$, i.e.\ in length, for any
	transverse displacement it produces; \eqref{eq:conenorm} is the quantitative form of
	this trade-off.
\end{remark}

\subsection{The quadratic defect of the first differential}
Given an admissible variation $\bv=(v_1,v)\in L^2([0,t],\R^d)$ we denote by
$\tau\mapsto q_{\bv}(\tau)$ the solution of
\begin{equation}\label{eq:perturbedode}
	\dot q(\tau)=X^\tau_{\bv(\tau)}\big(q(\tau)\big),\qquad q(0)=0 ,
\end{equation}
written in the coordinates of Remark~\ref{rmk:chart}(a), so that
$G^t_{q(t)}(\bv)=q_{\bv}(t)$.
\begin{lemma}[Localization]\label{lem:localization}
	Let $t\in(0,t_0]$ and let $\bv\in L^2([0,t],\R^d)$ satisfy the constraint
	\eqref{eq:conepointwise}. Then the Cauchy problem \eqref{eq:perturbedode} has a unique
	absolutely continuous solution on $[0,t]$ and, with $L_0$ and $\delta$ as in
	Lemma~\ref{lem:adapted}(v),
	\begin{equation}\label{eq:noleave}
		\left|q_{\bv}(\tau)\right|\le2L_0t<\delta,\qquad\tau\in[0,t].
	\end{equation}
	In particular $q_{\bv}(\tau)$ stays in the translated set $\mathcal B$.
\end{lemma}
\begin{proof}
	The right-hand side of \eqref{eq:perturbedode} is measurable in $\tau$, smooth in $x$ on
	$\mathcal B$ with derivatives bounded uniformly by Lemma~\ref{lem:adapted}(iv), and
	bounded by $L_0|\bv(\tau)|$, so that a unique local solution exists. As long as the
	solution stays in the ball of radius $\delta$ centred at the origin, which is contained
	in the translated set $\mathcal B$ by Remark~\ref{rmk:chart}(a), Lemma~\ref{lem:conenorm}
	gives
	\[
		\left|q_{\bv}(\tau)\right|
		\le\int_0^\tau\left|X^\theta_{\bv(\theta)}\big(q_{\bv}(\theta)\big)\right|d\theta
		\le L_0\|\bv\|_{L^1([0,t],\R^d)}\le2L_0t<\delta ,
	\]
	where the last inequality is Lemma~\ref{lem:adapted}(v). Since the bound is strictly
	smaller than $\delta$, the solution cannot reach the boundary of that ball, and a
	standard continuation argument shows that it exists on the whole interval $[0,t]$.
\end{proof}

\begin{lemma}\label{lem:defect}
	There exists $C_Q>0$, independent of $t$, such that for every $t\in(0,t_0]$ and every
	variation $\bv\in L^2([0,t],\R^d)$ satisfying \eqref{eq:conepointwise} one has
	\begin{equation}\label{eq:defectest}
		\left|G^t_{q(t)}(\bv)-d_0G^t_{q(t)}(\bv)\right|\le C_Q\,t\,\|\bV\|_{\mathcal H} .
	\end{equation}
\end{lemma}
\begin{proof}
	Let $x^1,\dots,x^n$ denote the coordinate functions of Remark~\ref{rmk:chart}(a),
	centred at $q(t)$. For a fixed $m\in\{1,\dots,n\}$, the Volterra expansion
	\eqref{eq:Volterra_expansion} with $N=1$ yields
	\[
		x^m\big(q_{\bv}(t)\big)
		=\int_0^t\Big(0\circ X^{\tau_1}_{\bv(\tau_1)}\Big)x^m\,d\tau_1
		+\int_{\Sigma_2(t)}\Big(q_{\bv}(\tau_2)\circ X^{\tau_2}_{\bv(\tau_2)}
		\circ X^{\tau_1}_{\bv(\tau_1)}\Big)x^m\,d\tau_2d\tau_1 .
	\]
	Since the coordinate functions are linear, the first summand is precisely the $m$-th
	component of $d_0G^t_{q(t)}(\bv)$, and the second one is therefore the $m$-th component
	of the defect $G^t_{q(t)}(\bv)-d_0G^t_{q(t)}(\bv)$. To estimate it, set
	\[
		\Lambda_2:=\max_{m,i,j}\ \sup_{[0,t_0]^2\times\mathcal B}
		\left|\Big(X^{\tau_2}_jX^{\tau_1}_ix^m\Big)(x)\right| ,
	\]
	which is finite and independent of $t$ by Lemma~\ref{lem:adapted}(iv); recall that
	$q_{\bv}(\tau_2)\in\mathcal B$ by Lemma~\ref{lem:localization}. Enlarging the
	simplex $\Sigma_2(t)$ to the square $[0,t]^2$, which is legitimate because the
	integrand is bounded in modulus by a non-negative function, we get
	\[
		\left|\left(G^t_{q(t)}(\bv)-d_0G^t_{q(t)}(\bv)\right)^m\right|
		\le\Lambda_2\,d^2\,\left\|\bv\right\|_{L^1([0,t],\R^d)}^2 .
	\]
	By the Cauchy--Schwarz inequality and Lemma~\ref{lem:conenorm},
	$\|\bv\|_{L^1}^2\le t\|\bv\|_{L^2}^2\le2t\|\bV\|_{\mathcal H}$, so that
	\eqref{eq:defectest} holds with $C_Q:=2\sqrt n\,\Lambda_2\,d^2$, after summing the $n$
	components.
\end{proof}

\begin{remark}\label{rmk:defectnonhomog}
	Estimate \eqref{eq:defectest} is, once more, non-homogeneous: the left-hand side -- the
	quadratic defect of the first differential -- is a genuinely quadratic quantity
	in $\bv$, while the right-hand side is linear in $\|\bV\|_{\mathcal H}$. This is not a
	weakness but the whole point. The defect will be compared, in
	Proposition~\ref{prop:perturbedKer}, with the boundary values $V_i(t)$, which are
	homogeneous of degree one in the variation; a homogeneous quadratic bound would be
	useless there for large variations, whereas the small prefactor $t$ makes the
	comparison effective uniformly on the whole admissible set. As in Remark~\ref{rmk:nonhomog}, the
	gain lies in the prefactor, not in the exponent.
\end{remark}

\begin{proposition}\label{prop:perturbedKer}
	There exists $C_0>0$, independent of $t\in(0,t_0]$, such that every variation
	$\bv\in L^2([0,t],\R^d)$ satisfying \eqref{eq:conepointwise} obeys, for
	$i=1,\dots,d$,
	\begin{equation}\label{eq:perturbedkerest}
		|V_i(t)|\le\sqrt{C_0}\,\|V\|_{L^2([0,t],\R^{d-1})}
		+C_\eta\left|G^t_{q(t)}(\bv)\right|
		+C_\eta C_Q\,t\,\|\bV\|_{\mathcal H}.
	\end{equation}
\end{proposition}
\begin{proof}
	For $i=1,\dots,d$ define
	$C_i(t):=\int_0^t\sum_{j=2}^d\left|\left\langle\eta_i(t),\dot X_j^\tau(q(t))
	\right\rangle\right|^2d\tau$. The map
	$(t,\tau)\mapsto\left\langle\eta_i(t),\dot X_j^\tau(q(t))\right\rangle$ is continuous on
	the compact triangle $\{(t,\tau)\mid0\le\tau\le t\le t_0\}$: indeed $t\mapsto q(t)$ and
	$t\mapsto\eta_i(t)$ are continuous, while the transported fields and their time
	derivatives depend smoothly on the relevant variables by Lemma~\ref{lem:adapted}(iv).
	Therefore
	\begin{equation}\label{eq:C0-perturbedker}
		C_0:=\sup_{0\le t\le t_0}\max_{1\le i\le d}C_i(t)<\infty ,
	\end{equation}
	and $C_0$ depends only on the reference extremal, the adapted frame, the coframe and
	$t_0$, neither on $\bv$ nor on $t$.

	Integrating by parts in \eqref{eq:fdif} with respect to $\tau$ we obtain
	\begin{equation}\label{eq:first-differential-boundary}
		\begin{aligned}
			d_0G^t_{q(t)}(\bv)&=q(t)\circ\int_0^tX^\tau_{\bv(\tau)}\,d\tau
			=q(t)\circ X^t_{\bV(t)}-q(t)\circ\int_0^t\dot X^\tau_{\bV(\tau)}\,d\tau\\
			&=V_1(t)X_1(q(t))+\sum_{j=2}^dV_j(t)X_j(q(t))
			-\int_0^t\dot X^\tau_{V(\tau)}(q(t))\,d\tau ,
		\end{aligned}
	\end{equation}
	where we used $\bV(0)=0$, $X_i^t=X_i$ for every $i$, and \eqref{eq:X1_identity} to
	cancel $\dot X_1^\tau$ inside the integral. Fix $i\in\{1,\dots,d\}$; evaluating
	\eqref{eq:first-differential-boundary} against $\eta_i(t)$ and using \eqref{eq:coframe},
	\begin{equation}\label{eq:Vi-perturbedker}
		V_i(t)=\int_0^t\sum_{j=2}^dV_j(\tau)
		\left\langle\eta_i(t),\dot X_j^\tau(q(t))\right\rangle d\tau
		+\left\langle\eta_i(t),d_0G^t_{q(t)}(\bv)\right\rangle .
	\end{equation}
	The Cauchy--Schwarz inequality and \eqref{eq:C0-perturbedker} bound the first summand
	by $\sqrt{C_0}\,\|V\|_{L^2([0,t],\R^{d-1})}$. As for the second one,
	Lemma~\ref{lem:defect} and the triangle inequality give
	\begin{equation}\label{eq:dG-bound-perturbedker}
		\left|d_0G^t_{q(t)}(\bv)\right|\le\left|G^t_{q(t)}(\bv)\right|
		+C_Q\,t\,\|\bV\|_{\mathcal H} ,
	\end{equation}
	and $|\eta_i(t)|\le C_\eta$. Combining the two estimates yields
	\eqref{eq:perturbedkerest}.
\end{proof}

\begin{remark}
    Proposition~\ref{prop:perturbedKer} is the point at which the end-point constraint enters the analysis in a quantitative way. Indeed, the strong Legendre condition yields coercivity only on the bulk quantity $\|V\|_{L^2}$, while the reduced Hessian depends on the stronger pair $(V(t),V)$ and therefore on the boundary value $V(t)$ as well. Estimate~\eqref{eq:perturbedkerest} recovers the boundary values $V_i(t)$ from the end-point defect $G_{q(t)}^t(\bv)$, showing that the latter controls the missing component of the norm. In particular, for exact competitors satisfying $G_{q(t)}^t(\bv)=0$, the boundary term $V(t)$ becomes comparable to $\|V\|_{L^2}$ for sufficiently small $t$. This observation is what ultimately allows the Legendre coercivity to be upgraded from $\|V\|_{L^2}$ to the full norm $\|\bV\|_{\mathcal H}$ in Proposition~\ref{prop:perturbednorm} and Theorem~\ref{thm:perturbedcoercivity}.
\end{remark}

\subsection{Coercivity of the Hessian}
\begin{proposition}\label{prop:perturbednorm}
	There exist $t_1\in(0,t_0]$ and constants $C_1,C_2>0$, all independent of $t$, such
	that for every $t\in(0,t_1]$ and every $\bv$ satisfying \eqref{eq:conepointwise},
	\begin{equation}\label{eq:perturbednormequiv}
		\left\|\bV\right\|_{\mathcal H}^2
		\le C_1^2\,\|V\|_{L^2([0,t],\R^{d-1})}^2
		+C_2^2\,\left|G^t_{q(t)}(\bv)\right|^2 .
	\end{equation}
\end{proposition}
\begin{proof}
	Set $N:=\|\bV\|_{\mathcal H}$, $L:=\|V\|_{L^2([0,t],\R^{d-1})}$,
	$\theta:=C_\eta C_Q\,t$, $D:=\left|G^t_{q(t)}(\bv)\right|$ and $\beta:=C_\eta D$, so
	that \eqref{eq:perturbedkerest} reads $|V_i(t)|\le\sqrt{C_0}\,L+\theta N+\beta$ for
	$i=1,\dots,d$. Squaring and using $(x+y+z)^2\le3(x^2+y^2+z^2)$,
	\begin{equation}\label{eq:squaredendpoint}
		|V_i(t)|^2\le3\left(C_0L^2+\theta^2N^2+\beta^2\right),\qquad i=1,\dots,d .
	\end{equation}
	As in Lemma~\ref{lem:conenorm}, constraint \eqref{eq:cone} gives $\dot V_1\le0$ a.e.,
	so that $V_1$ is non-increasing with $V_1(0)=0$, whence $|V_1(\tau)|\le|V_1(t)|$ for
	every $\tau\in[0,t]$ and $\|V_1\|_{L^2([0,t])}^2\le t\,|V_1(t)|^2\le|V_1(t)|^2$, the
	last inequality because $t\le1$. Summing,
	\[
		N^2=\sum_{i=1}^d|V_i(t)|^2+\|V_1\|_{L^2}^2+L^2
		\le(d+1)\max_i|V_i(t)|^2+L^2 ,
	\]
	and by \eqref{eq:squaredendpoint}, with $A:=3(d+1)$,
	$N^2\le A\left(C_0L^2+\theta^2N^2+\beta^2\right)+L^2$. Choose $t_1\in(0,t_0]$ so small
	that $A\theta^2=AC_\eta^2C_Q^2t^2\le\frac12$ for every $t\in(0,t_1]$; absorbing
	$A\theta^2N^2$ into the left-hand side we obtain
	$N^2\le2\left(AC_0+1\right)L^2+2A\beta^2$, which is \eqref{eq:perturbednormequiv} with
	$C_1:=\sqrt{2(AC_0+1)}$ and $C_2:=C_\eta\sqrt{2A}$.
\end{proof}

\begin{lemma}[Transport of the Legendre form]\label{lem:transport}
	For every $\tau\in[0,t]$ and every $w\in\R^{d-1}$,
	\[
		\left\langle\lambda(t),\left[\dot X^{\tau}_{w},X^{\tau}_{w}\right](q(t))
		\right\rangle
		=\left\langle\lambda(\tau),\left[\dot X_{w},X_{w}\right](q(\tau))\right\rangle
		=Q^{\tau}(w).
	\]
\end{lemma}
\begin{proof}
	By definition $X^\tau=e^{(t-\tau)X_1}_*X$, and the pushforward by a diffeomorphism is
	an isomorphism of Lie algebras, so that
	$\left[\dot X^\tau_w,X^\tau_w\right]=e^{(t-\tau)X_1}_*\left[\dot X_w,X_w\right]$.
	Evaluating at $q(t)=e^{(t-\tau)X_1}(q(\tau))$ and recalling \eqref{eq:covectcurve},
	which gives $\lambda(\tau)=\big(e^{(t-\tau)X_1}\big)^*\lambda(t)$, the claim follows
	from the definition of pullback.
\end{proof}

\begin{theorem}\label{thm:perturbedcoercivity}
	Assume the strong Legendre condition \eqref{eq:slg1} along the abnormal reference
	trajectory, with constant $c>0$. Then there exist $t_2\in(0,t_1]$ and constants
	$\gamma_*>0$, $C_3>0$, all independent of $t$, such that for every $t\in(0,t_2]$ and
	every variation $\bv$ satisfying \eqref{eq:conepointwise},
	\begin{equation}\label{eq:perturbedcoercivity}
		\lambda(t)\widehat{H^t}\big(V(t),V\big)
		\ \ge\ \gamma_*\left\|\bV\right\|_{\mathcal H}^2
		-C_3\left|G^t_{q(t)}(\bv)\right|^2 .
	\end{equation}
\end{theorem}
\begin{proof}
	We keep the notation $N,L,D,\theta,\beta$ of Proposition~\ref{prop:perturbednorm} and
	decompose, following \eqref{eq:hessGnoV1} and \eqref{eq:extensHessian},
	$\lambda(t)\widehat{H^t}\big(V(t),V\big)=I_1(V)+I_2\big(V(t),V\big)$, where
	\[
		\begin{aligned}
			I_1(V)&:=\int_0^t\left\langle\lambda(t),
			\left[\dot X^{\tau}_{V(\tau)},X^{\tau}_{V(\tau)}\right](q(t))
			\right\rangle d\tau,\\
			I_2\big(V(t),V\big)&:=\int_0^t\left\langle\lambda(t),
			\left[X_{V(t)}+\int_0^{\tau}\dot X^{\theta}_{V(\theta)}\,d\theta,\,
			\dot X^{\tau}_{V(\tau)}\right](q(t))\right\rangle d\tau .
		\end{aligned}
	\]
	By Lemma~\ref{lem:transport} and the strong Legendre condition \eqref{eq:slg1},
	\[
		I_1(V)=\int_0^tQ^{\tau}\big(V(\tau)\big)\,d\tau\ge c\,L^2 ,
	\]
	an estimate which involves no constraint on $\bv$ whatsoever. As for $I_2$, let
	$K_0>0$ be a uniform bound, on $[0,t_0]^2$, for the operator norms of the bilinear forms
	$(a,b)\mapsto\langle\lambda(t),[X_a,\dot X^{\tau}_b](q(t))\rangle$ and
	$(a,b)\mapsto\langle\lambda(t),[\dot X^{\theta}_a,\dot X^{\tau}_b](q(t))\rangle$ on
	$\R^{d-1}\times\R^{d-1}$, which exists by Lemma~\ref{lem:adapted}(iv) and may be taken
	equal to the constant of Lemma~\ref{lem:boundedextension}. Splitting $I_2$ into its two summands and
	using $\|V\|_{L^1}\le\sqrt tL$ we get
	\[
		\left|I_2\big(V(t),V\big)\right|\le K_0\,|V(t)|\,\sqrt t\,L+K_0\,t\,L^2 .
	\]
	By Propositions~\ref{prop:perturbedKer} and~\ref{prop:perturbednorm} we have both
	$|V(t)|\le\sqrt d\max_{i\ge2}|V_i(t)|$ and, for $t\le t_1$, $N\le C_1L+C_2D$.
	Therefore
	\[
		|V(t)|\le\sqrt d\left(\sqrt{C_0}\,L+\theta N+\beta\right)
		\le\sqrt d\left(\sqrt{C_0}+\theta C_1\right)L
		+\sqrt d\left(\theta C_2+C_\eta\right)D=:c_4L+c_5D,
	\]
	where $c_4,c_5$ are independent of $t\in(0,t_1]$ because
	$\theta=C_\eta C_Qt\le C_\eta C_Q$. Hence
	\[
		\left|I_2\right|\le\left(K_0c_4\sqrt t+K_0t\right)L^2+K_0c_5\sqrt t\,D\,L .
	\]
	Choose $t_2\in(0,t_1]$ so small that $K_0c_4\sqrt t+K_0t\le\frac c2$ for every
	$t\in(0,t_2]$; then
	$\lambda(t)\widehat{H^t}\big(V(t),V\big)\ge\frac c2L^2-K_0c_5\,D\,L$. By the Young
	inequality $K_0c_5DL\le\frac c4L^2+\frac{(K_0c_5)^2}{c}D^2$ we obtain
	\begin{equation}\label{eq:coercivityinL}
		\lambda(t)\widehat{H^t}\big(V(t),V\big)\ \ge\ \frac c4L^2
		-\frac{(K_0c_5)^2}{c}D^2 .
	\end{equation}
	Finally, Proposition~\ref{prop:perturbednorm} gives
	$L^2\ge\big(N^2-C_2^2D^2\big)/C_1^2$, so that
	\[
		\lambda(t)\widehat{H^t}\big(V(t),V\big)\ \ge\ \frac{c}{4C_1^2}N^2
		-\left(\frac{cC_2^2}{4C_1^2}+\frac{(K_0c_5)^2}{c}\right)D^2 ,
	\]
	which is \eqref{eq:perturbedcoercivity} with $\gamma_*:=\frac{c}{4C_1^2}$ and
	$C_3:=\frac{cC_2^2}{4C_1^2}+\frac{(K_0c_5)^2}{c}$.
\end{proof}

\begin{remark}[A quantitative absence of conjugate points]\label{rmk:conjugate}
	Estimate \eqref{eq:coercivityinL} is the heart of the matter, and it deserves a
	geometric reading. The term $I_1$ is, by Lemma~\ref{lem:transport}, the integral of the
	Legendre form along the extremal: it is the only term of the second variation carrying
	a sign, and it is positive definite by hypothesis. All the remaining terms are
	length-three brackets paired with $\lambda(t)$ whose contribution carries a factor
	$\sqrt t$ or $t$, coming from H\"older's inequality on an interval of length $t$. Short
	time therefore plays here the role that the absence of conjugate points plays in the
	classical Jacobi theory: it guarantees that the indefinite part of the second variation
	has not yet had the time to overcome the definite one. What the argument does
	not provide is the sharp threshold, i.e.\ the first conjugate time.
\end{remark}

\section{Separation, rigidity and short-time length minimality}\label{sec:separation}

\subsection{The third-order expansion along the abnormal direction}
We now expand the abnormal component of the end-point map to third order. Two remarks
guide the statement. First, by Remark~\ref{rmk:notdiscarded} the second-order part of the
expansion contains, besides the reduced Hessian \eqref{eq:hessGnoV1}, the contribution of
the genuinely second-order operator in \eqref{eq:Second_order_differential}, which does
not vanish outside $\ker d_0G^t_{q(t)}$ and which we isolate as the term $S^t(\bv)$
below. Second, by Remark~\ref{rmk:nonhomog}, the third-order remainder can be made
quadratic in $\|\cdot\|_{\mathcal H}$ with a prefactor $\sqrt t$, provided exactly one
factor of the control is converted into its primitive.

\begin{lemma}[Third-order expansion]\label{lem:cubicremainder}
	There exist constants $C_R,C_S>0$, independent of $t\in(0,t_2]$, such that the
	following holds. Let $t\in(0,t_2]$ and let $\bv\in L^2([0,t],\R^d)$ satisfy the
	admissibility constraint \eqref{eq:conepointwise}. Then
	\begin{equation}\label{eq:exactexpansion}
		a_t\big(G^t_{q(t)}(\bv)\big)
		=\frac12\,\lambda(t)\widehat{H^t}\big(V(t),V\big)+S^t(\bv)+R^t(\bv),
	\end{equation}
	where $a_t=\left\langle\ell_t,\cdot\right\rangle$ is the affine abnormal coordinate of
	Lemma~\ref{lem:adapted}(iii),
	\begin{equation}\label{eq:Sterm}
		S^t(\bv):=\frac12\left(q(t)\circ W_{\bv}\circ W_{\bv}\right)a_t ,
		\qquad
		W_{\bv}:=\int_0^tX^\tau_{\bv(\tau)}\,d\tau ,
	\end{equation}
	and
	\begin{equation}\label{eq:remainderintegral}
		R^t(\bv):=\int_{\Sigma_3(t)}\Big(q_{\bv}(\tau_3)\circ X_{\bv(\tau_3)}^{\tau_3}
		\circ X_{\bv(\tau_2)}^{\tau_2}\circ X_{\bv(\tau_1)}^{\tau_1}\Big)a_t
		\,d\tau_3d\tau_2d\tau_1 .
	\end{equation}
	Moreover
	\begin{equation}\label{eq:remainderestimate}
		\left|R^t(\bv)\right|\le C_R\sqrt t\,\|\bV\|_{\mathcal H}^2,
		\qquad
		\left|S^t(\bv)\right|\le C_S\,\|\bV\|_{\mathcal H}
		\left(\left|G^t_{q(t)}(\bv)\right|+C_Q\,t\,\|\bV\|_{\mathcal H}\right).
	\end{equation}
\end{lemma}
\begin{proof}
	With $L_0$ and $\delta$ as in Lemma~\ref{lem:adapted}(v) one has
	$\left|X^\tau_{\bw}(x)\right|\le L_0|\bw|$ for every $\bw\in\R^d$ and $x\in\mathcal B$,
	and by Lemma~\ref{lem:localization} the trajectory $q_{\bv}$ stays in $\mathcal B$ for
	every $\tau\in[0,t]$. The Volterra expansion~\eqref{eq:Volterra_expansion}
	with $N=2$, applied to the function $a_t$ and evaluated at $q(t)=0$, yields
	\begin{equation}\label{eq:volterra-at}
		a_t\big(G^t_{q(t)}(\bv)\big)
		=\int_0^t\Big(X_{\bv(\tau_1)}^{\tau_1}a_t\Big)(q(t))\,d\tau_1
		+\int_{\Sigma_2(t)}\Big(X_{\bv(\tau_2)}^{\tau_2}\circ X_{\bv(\tau_1)}^{\tau_1}
		a_t\Big)(q(t))\,d\tau_2d\tau_1+R^t(\bv),
	\end{equation}
	with $R^t(\bv)$ as in \eqref{eq:remainderintegral}. Because $d_{q(t)}a_t=\lambda(t)$
	and $a_t$ is affine, the first term equals
	$\left\langle\lambda(t),d_0G^t_{q(t)}(\bv)\right\rangle$, which vanishes by
	abnormality. The second term is $\frac12\big(d_0^2G^t_{q(t)}(\bv)\big)a_t$, and by
	\eqref{eq:Second_order_differential} it splits into the bracket part and the
	second-order part:
	\[
		\frac12\int_{\Sigma_2(t)}\left\langle\lambda(t),
		\left[X^{\tau_2}_{\bv(\tau_2)},X^{\tau_1}_{\bv(\tau_1)}\right](q(t))\right\rangle
		d\tau_2d\tau_1
		+\frac12\left(q(t)\circ W_{\bv}\circ W_{\bv}\right)a_t .
	\]
	By \eqref{eq:cruc_vanishing}, \eqref{eq:hessGnoV1} and \eqref{eq:extensHessian}, the
	first summand is $\frac12\lambda(t)\widehat{H^t}\big(V(t),V\big)$, and the second one
	is $S^t(\bv)$. This proves \eqref{eq:exactexpansion}.  Since $a_t$ is affine, for any smooth vector field
	$W$ on $\mathcal B$ one has $\left(W\circ W\right)a_t
	=\left\langle\ell_t,D_xW\cdot W\right\rangle$. Integrating by parts as in
	\eqref{eq:first-differential-boundary} and using \eqref{eq:X1_identity},
	\[
		W_{\bv}=X^t_{\bV(t)}-\int_0^t\dot X^\tau_{V(\tau)}\,d\tau ,
	\]
	so that, by Lemma~\ref{lem:adapted}(iv), there is $\Lambda_3>0$ with
	\[
		\sup_{x\in\mathcal B}\left\|D_xW_{\bv}(x)\right\|
		\le\Lambda_3\left(|\bV(t)|+\|V\|_{L^1}\right)
		\le2\Lambda_3\,\|\bV\|_{\mathcal H},
	\]
	where we used $\|V\|_{L^1}\le\sqrt t\|V\|_{L^2}\le\|\bV\|_{\mathcal H}$. On the other
	hand $W_{\bv}(q(t))=d_0G^t_{q(t)}(\bv)$ by \eqref{eq:fdif}, whence by
	\eqref{eq:dG-bound-perturbedker}
	\[
		\left|S^t(\bv)\right|\le\frac12\,C_\ell\left\|D_xW_{\bv}\right\|
		\left|d_0G^t_{q(t)}(\bv)\right|
		\le C_S\|\bV\|_{\mathcal H}\left(\left|G^t_{q(t)}(\bv)\right|
		+C_Q\,t\,\|\bV\|_{\mathcal H}\right)
	\]
	with $C_S:=C_\ell\Lambda_3$. For $i,j,k\in\{1,\dots,d\}$ define the
	coefficients
	\[
		c_{kji}^t(\tau_3,\tau_2,\tau_1;x):=\Big(x\circ X_k^{\tau_3}\circ X_j^{\tau_2}
		\circ X_i^{\tau_1}\Big)a_t .
	\]
	Since $a_t$ is affine, its coordinate derivatives of order at least two vanish; thus
	$c^t_{kji}$ involves only $\ell_t$, the coordinate expressions of the transported
	fields, and their spatial derivatives up to order two; consequently
	$\partial_{\tau_3}c^t_{kji}$ and $D_xc^t_{kji}$ involve derivatives of the transported
	fields with respect to $x$ and $\tau$ of total order at most three. By
	Lemma~\ref{lem:adapted}(iii)--(iv) there is $\Lambda>0$ with
	\begin{equation}\label{eq:cubic-coefficient-bound}
		\left|c_{kji}^t\right|+\left|\partial_{\tau_3}c_{kji}^t\right|
		+\left|D_xc_{kji}^t\right|\le\Lambda
	\end{equation}
	for every $t\in(0,t_2]$, $x\in\mathcal B$ and $0\le\tau_3\le\tau_2\le\tau_1\le t$.
	Expanding the three perturbation fields in \eqref{eq:remainderintegral},
	\begin{equation}\label{eq:remainder-components}
		R^t(\bv)=\sum_{i,j,k=1}^d\int_{\Sigma_3(t)}
		v_k(\tau_3)v_j(\tau_2)v_i(\tau_1)\,
		c_{kji}^t\big(\tau_3,\tau_2,\tau_1;q_{\bv}(\tau_3)\big)\,d\tau_3d\tau_2d\tau_1 .
	\end{equation}
	Estimating \eqref{eq:remainder-components} directly by $\|\bv\|_{L^1}^3$ would only
	give a term of order $\|\bV\|_{\mathcal H}^{3/2}$. We therefore integrate by parts in
	$\tau_3$, the innermost variable of the simplex $0\le\tau_3\le\tau_2\le\tau_1\le t$.
	For fixed $\tau_1,\tau_2,i,j,k$ put
	$h_{kji}(\tau):=c_{kji}^t\big(\tau,\tau_2,\tau_1;q_{\bv}(\tau)\big)$; this map is
	absolutely continuous and, by \eqref{eq:perturbedode}, \eqref{eq:pointwisecone} and
	\eqref{eq:cubic-coefficient-bound},
	\begin{equation}\label{eq:hprime}
		\left|h_{kji}(\tau)\right|\le\Lambda,
		\qquad
		\left|h_{kji}'(\tau)\right|\le\left|\partial_{\tau_3}c^t_{kji}\right|
		+\left|D_xc^t_{kji}\right|\,\left|X^\tau_{\bv(\tau)}(q_{\bv}(\tau))\right|
		\le\Lambda\left(1+2L_0\right)=:\Lambda_1 ,
	\end{equation}
	uniformly in $\tau_1,\tau_2,\tau$ and in $t\in(0,t_2]$. Since $V_k(0)=0$,
	\[
		\int_0^{\tau_2}h_{kji}(\tau_3)v_k(\tau_3)\,d\tau_3
		=h_{kji}(\tau_2)V_k(\tau_2)-\int_0^{\tau_2}h_{kji}'(\tau_3)V_k(\tau_3)\,d\tau_3 ,
	\]
	so that $R^t(\bv)=A^t(\bv)+B^t(\bv)$, where
	\[
		\begin{aligned}
			A^t(\bv)&:=\sum_{i,j,k}\int_{\Sigma_2(t)}h_{kji}(\tau_2)\,V_k(\tau_2)\,
			v_j(\tau_2)\,v_i(\tau_1)\,d\tau_2d\tau_1,\\
			B^t(\bv)&:=-\sum_{i,j,k}\int_{\Sigma_3(t)}h_{kji}'(\tau_3)\,V_k(\tau_3)\,
			v_j(\tau_2)\,v_i(\tau_1)\,d\tau_3d\tau_2d\tau_1 .
		\end{aligned}
	\]
	Enlarging the simplices to the corresponding cubes and using \eqref{eq:hprime},
	\[
		\begin{aligned}
			\left|A^t(\bv)\right|&\le\Lambda d^3\int_0^t\!\!\int_0^{t}
			\left|\bV(\tau_2)\right|\left|\bv(\tau_2)\right|\left|\bv(\tau_1)\right|
			d\tau_2d\tau_1\\
			&\le\Lambda d^3\,\|\bV\|_{L^2}\|\bv\|_{L^2}\|\bv\|_{L^1}
			\le\Lambda d^3\sqrt t\,\|\bV\|_{L^2}\|\bv\|_{L^2}^2,\\
			\left|B^t(\bv)\right|&\le\Lambda_1d^3\,\|\bV\|_{L^1}\|\bv\|_{L^1}^2
			\le\Lambda_1d^3\,t^{3/2}\|\bV\|_{L^2}\|\bv\|_{L^2}^2
			\le\Lambda_1d^3\sqrt t\,\|\bV\|_{L^2}\|\bv\|_{L^2}^2 ,
		\end{aligned}
	\]
	the last inequality because $t\le1$. Finally, Lemma~\ref{lem:conenorm} gives
	$\|\bv\|_{L^2}^2\le2\|\bV\|_{\mathcal H}$, while trivially
	$\|\bV\|_{L^2}\le\|\bV\|_{\mathcal H}$; therefore
	$\left|R^t(\bv)\right|\le C_R\sqrt t\,\|\bV\|_{\mathcal H}^2$ with
	$C_R:=2(\Lambda+\Lambda_1)d^3$, which is \eqref{eq:remainderestimate}.
\end{proof}

\subsection{The separation theorem}
\begin{theorem}\label{thm:separation}
	Assume the Goh condition \eqref{eq:Goh_section} and the strong Legendre condition
	\eqref{eq:slg1} along the abnormal reference trajectory. Then there exist
	$t_*\in(0,t_2]$ and constants $\eta>0$, $C_4>0$, all independent of $t$, such that for
	every $t\in(0,t_*]$ and every variation $\bv\in L^2([0,t],\R^d)$ satisfying the
	admissibility constraint \eqref{eq:conepointwise},
	\begin{equation}\label{eq:separation}
		\left\langle\ell_t,G^t_{q(t)}(\bv)\right\rangle
		\ \ge\ \eta\left\|\bV\right\|_{\mathcal H}^2
		-C_4\left|G^t_{q(t)}(\bv)\right|^2 .
	\end{equation}
	In particular,
	\begin{equation}\label{eq:separationdist}
		C_\ell\left|G^t_{q(t)}(\bv)\right|+C_4\left|G^t_{q(t)}(\bv)\right|^2
		\ \ge\ \eta\left\|\bV\right\|_{\mathcal H}^2 .
	\end{equation}
\end{theorem}
\begin{proof}
	Write $N:=\|\bV\|_{\mathcal H}$ and $D:=\left|G^t_{q(t)}(\bv)\right|$. By
	Lemma~\ref{lem:cubicremainder} and Theorem~\ref{thm:perturbedcoercivity},
	\[
		\left\langle\ell_t,G^t_{q(t)}(\bv)\right\rangle
		=\frac12\lambda(t)\widehat{H^t}\big(V(t),V\big)+S^t(\bv)+R^t(\bv)
		\ \ge\ \frac{\gamma_*}2N^2-\frac{C_3}2D^2
		-\left|S^t(\bv)\right|-\left|R^t(\bv)\right| .
	\]
	By \eqref{eq:remainderestimate},
	$\left|R^t(\bv)\right|+\left|S^t(\bv)\right|
	\le\left(C_R\sqrt t+C_SC_Qt\right)N^2+C_SND$, and by the Young inequality
	$C_SND\le\frac{\gamma_*}8N^2+\frac{2C_S^2}{\gamma_*}D^2$. Hence
	\[
		\left\langle\ell_t,G^t_{q(t)}(\bv)\right\rangle
		\ \ge\ \left(\frac{3\gamma_*}8-C_R\sqrt t-C_SC_Qt\right)N^2
		-\left(\frac{C_3}2+\frac{2C_S^2}{\gamma_*}\right)D^2 .
	\]
	Choosing $t_*\in(0,t_2]$ so small that $C_R\sqrt t+C_SC_Qt\le\frac{\gamma_*}8$ for
	every $t\in(0,t_*]$, we obtain \eqref{eq:separation} with $\eta:=\frac{\gamma_*}4$ and
	$C_4:=\frac{C_3}2+\frac{2C_S^2}{\gamma_*}$. Inequality \eqref{eq:separationdist}
	follows from $\left\langle\ell_t,x\right\rangle\le|\ell_t|\,|x|\le C_\ell|x|$.
\end{proof}

Estimate \eqref{eq:separationdist} is the quantitative form of the separation: a non-zero
admissible variation cannot bring the end-point back to $q(t)$, and, more precisely, the
end-point displacement it produces is at least of the order of the square of the weak
norm of its primitive. In the chart, \eqref{eq:separation} says that the image of the
admissible set under the nonlinear end-point map is contained in the complement of the
open ball $\{x\in\R^n:\ a_t(x)<-C_4|x|^2\}$, whose boundary passes through $q(t)$ and is
tangent there to the hyperplane $\{a_t=0\}$. Thus the hyperplane $\{a_t=0\}$ is not, in
general, a supporting hyperplane of the image: the image is separated from $q(t)$ only up
to the quadratic correction $C_4|x|^2$, which is enough to conclude because an exact
competitor has $x=0$.

\begin{remark}[Why $\lambda(t)$ is the only available covector]\label{rmk:onlyGoh}
	The asymmetry between $\lambda(t)$ and the remaining elements of
	$\Lambda_t=\left(\im d_0G^t_{q(t)}\right)^\perp$ is not a matter of convenience: it is
	built into the hypotheses, as anticipated in Subsection~\ref{ssec:onedir}. The Goh
	condition asserts the existence of at least one abnormal lift annihilating all
	the brackets $[X_i,X_j]$; it does not assert that every element of $\Lambda_t$ does the
	same. For $\mu\in\Lambda_t$ not satisfying Goh, the reduced Hessian
	$\mu\widehat{H^t}$ is not even well defined by \eqref{eq:extensHessian}, since the
	elimination of $V_1$ rests on \eqref{eq:cruc_vanishing}. An exact competitor, on the
	other hand, satisfies $\langle\mu,G^t_{q(t)}(\bv)\rangle=0$ for every
	$\mu\in\Lambda_t$, in particular for $\lambda(t)$: a strictly positive lower bound for
	the $\lambda(t)$-component alone already produces the contradiction.
\end{remark}

\begin{corollary}[Rigidity]\label{cor:rigidity}
	Let $t\in(0,t_*]$ and let $\bv$ satisfy \eqref{eq:conepointwise} together with
	$G^t_{q(t)}(\bv)=0$, i.e.\ $F^t_q(u+\bv)=q(t)$. Then $\bv\equiv0$.
\end{corollary}
\begin{proof}
	The left-hand side of \eqref{eq:separation} vanishes and so does its last term, whence
	$\eta\|\bV\|_{\mathcal H}^2\le0$ and therefore $\|\bV\|_{\mathcal H}=0$. Since
	$\|\cdot\|_{\mathcal H}$ dominates the $L^2$-norm of $\bV$, we get $\bV\equiv0$ and
	thus $\bv=\dot{\bV}\equiv0$.
\end{proof}

\subsection{Short-time local length minimality}
It remains to convert Corollary~\ref{cor:rigidity} into a statement about the length
functional \eqref{eq:length}. Two points are worth a remark here.

First, the reduction to the admissible set \eqref{eq:conepointwise} was obtained through the
constant-speed reparametrization of Lemma~\ref{lem:reparam}, which identifies curves
differing by a reparametrization. Consequently the
conclusion $\bv\equiv0$ concerns the reparametrized competitor and yields that
$\gamma_{\tilde u}$ is a reparametrization of the reference curve; the natural
formulation of the result is therefore uniqueness up to reparametrization. Strictness
cannot hold on the space of parametrized controls: in the adapted frame, for every
nondecreasing absolutely continuous surjection $\varphi$ of $[0,t]$ the control
$\dot\varphi e_1$ produces the curve $\gamma_u|_{[0,t]}\circ\varphi$, which has the same
end-points and the same length as the reference arc, and $\|\dot\varphi e_1-u\|_{L^2}$ can
be arbitrarily small without $\dot\varphi$ being identically equal to $1$. Strictness
holds on the quotient by reparametrizations, and on the set of controls of constant speed.

Second, the size of the neighborhood on which the end-point map is well defined must be
quantified. Here the admissible set works in our favor: by Lemma~\ref{lem:conenorm} every
admissible variation satisfies $\|\bv\|_{L^1}\le2t$, so that competitors are automatically
small in $L^1$ for small $t$ and, by \eqref{eq:noleave}, the associated trajectories
$q_{\bv}$ never leave the ball of radius $\delta$. Equivalently, and more intrinsically, a
horizontal curve issuing from $q$ of length at most $t$ is contained in the closed
sub-Riemannian ball $\overline B(q,t)$, which is contained in $\mathcal O_1$ for $t$
small. No further smallness assumption on the competitor is required. Recall that, by the
H\"ormander condition, the sub-Riemannian distance is finite and induces the manifold
topology. Hence, since $\mathcal O_1$ is an open neighborhood of $q$, there exists
$\rho>0$ such that $B(q,\rho)\subset\mathcal O_1$, and therefore
$\overline B(q,t)\subset\mathcal O_1$ for every $t<\rho$.

\begin{theorem}[Short-time length minimality up to reparametrization]\label{thm:minimality}
	Let $u\in L^2([0,1],\R^d)$ be an arclength parametrized strictly singular control, and
	assume that the Goh condition \eqref{eq:Goh_section} and the strong Legendre condition
	\eqref{eq:strongLeg} hold along the corresponding nonzero abnormal lift
	$t\mapsto\lambda(t)$, with constant $c>0$. Then there exists $t_*\in(0,1]$ such that,
	for every $t\in(0,t_*]$, the restriction $u|_{[0,t]}$ minimizes the length
	\eqref{eq:length} among all admissible curves joining $q$ to $q(t)$: if
	$\gamma_{\tilde u}$ is an admissible curve with $F^t_q(\tilde u)=q(t)$ and
	$\ell(\gamma_{\tilde u})\le\ell(\gamma_{u^t})=t$, then $\ell(\gamma_{\tilde u})=t$ and
	$\gamma_{\tilde u}=\gamma_u|_{[0,t]}\circ\varphi$ for a nondecreasing absolutely
	continuous surjection $\varphi$ of $[0,t]$, that is, $\gamma_{\tilde u}$ coincides with
	$\gamma_u|_{[0,t]}$ up to reparametrization. If, in addition, $\tilde u$ has constant
	speed, then $\tilde u=u$. In particular $u|_{[0,t]}$ is the unique length-minimizing
	control of constant speed joining $q$ to $q(t)$.
\end{theorem}
\begin{proof}
	Let $t\in(0,t_*]$, with $t_*$ as in Theorem~\ref{thm:separation}, possibly further
	reduced so that $\overline B(q,t)\subset\mathcal O_1$, and let $\gamma_{\tilde u}$ be a
	competitor as in the statement. Since $\ell(\gamma_{\tilde u})\le t$, every point of
	$\gamma_{\tilde u}$ has sub-Riemannian distance at most $t$ from $q$, so that
	$\gamma_{\tilde u}\subset\overline B(q,t)\subset\mathcal O_1\subset\mathcal O$. Hence
	$\gamma_{\tilde u}$ is represented in the adapted frame of Lemma~\ref{lem:adapted} by a
	control $\tilde u\in L^2([0,t],\R^d)$, with the same length because both frames are
	orthonormal, and the end-point map is well defined on it. As observed in
	Subsection~\ref{ssec:adapted}, $\Len(\tilde u)>0$. Apply
	Lemma~\ref{lem:reparam} and let $\tilde u_\varphi$ be the control of the curve $\eta$
	there. By the lemma
	$\left|\tilde u_\varphi(\tau)\right|=\ell(\gamma_{\tilde u})/t\le1$ for a.e.\ $\tau$,
	and the reparametrized curve has the same end-point as $\gamma_{\tilde u}$, so that
	$F^t_q(\tilde u_\varphi)=q(t)$. Writing
	$\bv:=\tilde u_\varphi-u=(v_1,v)$ in the splitting \eqref{eq:boldnotation} and
	recalling $u=(1,0)$, the pointwise bound above reads
	\[
		\big(1+v_1(\tau)\big)^2+|v(\tau)|^2=\left|\tilde u_\varphi(\tau)\right|^2\le1,
		\qquad\text{a.e. }\tau\in[0,t],
	\]
	which is exactly \eqref{eq:conepointwise}. Moreover $G^t_{q(t)}(\bv)=0$ by
	\eqref{eq:SLCEndpoint}. Corollary~\ref{cor:rigidity} then gives $\bv\equiv0$, i.e.\
	$\tilde u_\varphi=u$ almost everywhere on $[0,t]$. In particular
	$\left|\tilde u_\varphi\right|\equiv1$, whence $\ell(\gamma_{\tilde u})=t$, and
	$\gamma_{\tilde u}=\gamma_{\tilde u_\varphi}\circ\varphi
	=\gamma_u|_{[0,t]}\circ\varphi$ is a reparametrization of the reference arc. If
	$\tilde u$ has constant speed, then $\dot\varphi\equiv1$ and $\varphi$ is the identity,
	because $\varphi(0)=0$, so that $\tilde u=u$.
\end{proof}

\subsection{On the strength of the results}\label{ssec:strength}
We close by discussing what Theorems~\ref{thm:separation} and~\ref{thm:minimality} assert, and what are their current limitations.

(a) The topology. The minimality statement is not confined to a small
neighborhood of the reference control in any topology on controls: the only restriction
imposed on the competitor is that it join $q$ to $q(t)$ with length at most $t$, which
already confines it to the closed ball $\overline B(q,t)$. This is a consequence of the
fact that the admissibility constraint \eqref{eq:conepointwise} is a pointwise,
non-perturbative condition, and that all the estimates of
Sections~\ref{sec:hessian} and~\ref{sec:separation} are uniform on the admissible set
rather than asymptotic as $\bv\to0$. In particular the length of every control $\tilde u$
with $F^t_q(\tilde u)=q(t)$ is at least $t$, with equality only for reparametrizations of
$u|_{[0,t]}$; this gives the minimality of $u|_{[0,t]}$ with respect to the
$L^2$-topology on controls (equivalently, the $W^{1,2}$-topology on trajectories in the
adapted chart), which is stronger than the $L^\infty$-minimality of \cite{AS96} and than
the rank-two $L^2$ result of \cite{AS95}. Because of reparametrizations, strictness holds
on the quotient by reparametrizations and on the set of controls of constant speed, not
among parametrized controls. The conclusion is accompanied by rigidity: the competitor is
not merely no shorter, it is the same curve up to reparametrization.

(b) What is genuinely local is the time. The conclusion holds for $t\le t_*$, and
$t_*$ is determined, through the finitely many smallness requirements imposed in
Proposition~\ref{prop:perturbednorm} and in
Theorems~\ref{thm:perturbedcoercivity} and~\ref{thm:separation}, by the constants $c$,
$C_0$, $C_Q$, $C_R$, $C_S$, $C_\eta$, $C_\ell$, $K_0$, $L_0$, $\delta$, $\Lambda$,
$\Lambda_1$, $\Lambda_2$ and $\Lambda_3$, and by the choice of $t_0$ in
Lemma~\ref{lem:adapted}. All of these are suprema, over a compact neighborhood of the
reference arc, of the transported fields and of their derivatives up to order three, so
$t_*$ is in principle computable; it is
however certainly not sharp, and the argument gives no information about the analogue of
the first conjugate time (Remark~\ref{rmk:conjugate}).

(c) The hypotheses. Coercivity is assumed only at the Lebesgue points of $u$, which
is weaker than the pointwise assumption of \cite{AS95,AS96}; the passage to all times is
part of the conclusion (Proposition~\ref{prop:ucont}). No assumption is made on $n$, on
the corank of $d_0G^t_{q(t)}$, or on the growth vector of the
distribution: in rank two the hypothesis implies niceness, in higher rank it replaces it
by the maximal-rank condition \eqref{eq:onedim} (Remark~\ref{rmk:niceanalogy}). On the
other hand the hypothesis is a genuine restriction, and the interesting degenerate cases
ii) and iii) listed after Proposition~\ref{prop:essmin} -- where the Legendre form
vanishes somewhere, respectively on a set of positive measure -- remain open. There the
mechanism exploited here breaks down at its root, since $I_1$ no longer dominates $I_2$
even for short times.

(d) The mechanism, in one line. The entire argument rests on the non-homogeneous
inequality $\|\bv\|_{L^2}^2\le2\|\bV\|_{\mathcal H}$ of Lemma~\ref{lem:conenorm}, valid on
the admissible set and false on any linear subspace: it converts a cubic quantity into a
quadratic one with a small prefactor, and it is what allows the second-order term to
dominate all the others uniformly on the whole admissible set, with no smallness
requirement on the variation.

\bibliographystyle{abbrv}
\bibliography{Biblio}

\end{document}